\documentclass[a4paper, 11pt]{article}

\usepackage{amsmath,amssymb}
\usepackage{amsthm}
\usepackage[all]{xy}
\usepackage[dvipdfmx]{graphicx}
\usepackage{tikz}
\usepackage{comment}

\usepackage[top=30truemm,bottom=30truemm,left=25truemm,right=25truemm]{geometry}

\theoremstyle{definition}
\newtheorem{lemma}{Lemma}[section]
\newtheorem{definition}[lemma]{Definition}
\newtheorem{theorem}[lemma]{Theorem}
\newtheorem{example}[lemma]{Example}
\newtheorem{proposition}[lemma]{Proposition}
\newtheorem{corollary}[lemma]{Corollary}
\newtheorem{notation}[lemma]{Notation}
\newtheorem{remark}[lemma]{Remark}

\newtheorem{construction}[lemma]{Construction}
\newtheorem{problem}[lemma]{Problem}

\title{\bf On the boundary components of central streams}
\author{Nobuhiro Higuchi}

\begin{document}

\maketitle

\begin{abstract}
Foliations on the space of $p$-divisible groups were
studied by Oort in 2004.
In his theory, special leaves called central stream
play an important role.
In this paper, we give a complete classification of the boundary components
of the central streams for an arbitrary Newton polygon in the unpolarized case.
Hopefully this classification would help us to know 
the boundaries of other leaves and more detailed structure of
the boundaries of central streams.
\end{abstract}

\section{Introduction} \label{Intro}

In \cite{oortFoliations}, Oort defined the notion of
\textit{leaves} on a family of $p$-divisible groups,
which are often called Barsotti-Tate groups,  
to study the moduli space of abelian varieties
in positive characteristic.
Let $k$ be an algebraically closed field of characteristic $p$.
Let ${\rm S}$ be a noetherian scheme over $k$.
For a $p$-divisible group $Y$ over $k$,
Oort introduced in \cite[2.1]{oortFoliations}
a locally closed subset $\mathcal C_Y({\rm S})$
for a $p$-divisible group $\mathcal Y$ over ${\rm S}$ characterized by 
$s$ belongs to $\mathcal C_Y({\rm S})$ if and only if $\mathcal Y_s$ 
is isomorphic to $Y$ over an algebraically closed field 
containing $k(s)$ and $k$;
see the first paragraph of Section~\ref{Definitions} for a review.
We call $\mathcal C_Y({\rm S})$ the central leaf associated to $Y$ and $\mathcal Y$,
if $\mathcal Y \to {\rm S}$ is a universal family 
over a deformation space or a moduli space.

In \cite[2.2]{oortFoliations} Oort showed that
$\mathcal C_Y(\rm S)$ is closed in an open Newton polygon stratum.
We consider $\mathcal C_Y(\rm S)$ as a locally closed subscheme
of $\rm S$ by giving the induced reduced scheme structure.  
We are interested in the boundaries of leaves on the deformation space,
the problem of which will be formulated from the next paragraph. 
See \cite[6.10]{oortFoliations} for a question in the polarized case
(i.e., the case that $p$-divisible groups associated with polarized abelian varieties).

Let us formulate the problem on the boundaries of central leaves.
Fix a $p$-divisible group $X_0$ over $k$.
Let $\mathfrak{Def}(X_0) = {\rm Spf} (\Gamma)$
be the deformation space of $X_0$.
The deformation space is the formal scheme pro-representing 
the functor ${\rm Art}_k \to {\rm Sets}$
which sends $R$ to the set of isomorphism classes of 
$p$-divisible groups $X$ over $R$ such that $X_k \simeq X_0$.
Here ${\rm Art}_k$ denotes the category of local Artinian rings
with residue field $k$.
Let $\mathfrak X' \to {\rm Spf}(\Gamma)$ be the universal $p$-divisible group.
In \cite[2.4.4]{JongCristalline} de Jong proved that
the category of $p$-divisible groups over ${\rm Spf}(\Gamma)$
is equivalent to the category of the $p$-divisible groups
over ${\rm Spec} (\Gamma) =: {\rm Def}(X_0)$.
Let $\mathfrak X$
be the $p$-divisible group over ${\rm Def}(X_0)$
obtained from $\mathfrak X'$ by this equivalence.
Oort studied $\mathcal C_{X_0}({\rm Def}(X_0))$, see \cite[2.7]{oortFoliations}.
We are interested in $\mathcal C_Y({\rm Def}(X))$ for $X \neq Y$
with $\mathcal Y = \mathfrak X$.
Here is a basic problem:
\begin{problem} \label{ProblemOfGeneralCase}
Let $Y$ be a $p$-divisible group over $k$.
Classify $p$-divisible groups $X$ over $k$ such that
$\mathcal C_Y({\rm Def}(X)) \neq \emptyset$.
Here $\mathcal C_Y({\rm Def}(X)) \neq \emptyset$ means that
$X$ appears as a specialization of a family of $p$-divisible groups
whose geometric fibers are isomorphic to $Y$. 
\end{problem}

In this paper we discuss the case that 
the $p$-divisible group $Y$ is ``minimal",
as the general case looks difficult.
Oort introduced the notion of minimal $p$-divisible groups
in \cite[1.1]{oortMinimal}.
Oort showed in \cite[1.2]{oortMinimal} that
the property:
{\it Let $X$ be a minimal $p$-divisible group over $k$,
and let $Y$ be a $p$-divisible group over $k$.
If $X[p] \simeq Y[p]$, then $X \simeq Y$,
where, $X[p]$ is the $p$-kernel of $p$-multiplication.}
For a Newton polygon $\xi$,
we obtain the minimal $p$-divisible group $H(\xi)$.
See the third and fourth paragraphs of Section~\ref{Definitions}
for the definitions of Newton polygons and minimal $p$-divisible groups. 

For $\mathcal C_Y({\rm Def}(X_0))$,
if $Y$ is minimal, 
then we call it a \textit{central stream}.
This notion is a ``central" tool in the theory of foliations.
For instance, it is known that the difference between
central leaves and central streams comes from 
isogenies of $p$-divisible groups.
Thus to study boundaries of general leaves, 
it is natural to start with investigating boundaries of central streams.

Let $\xi$ be a Newton polygon.
For the notation as above,
we may treat the problem:
\begin{problem} \label{ProblemOfCHxiDefXneqEmpty}
Classify $p$-divisible groups $X$ over $k$ such that 
$\mathcal C_{H(\xi)}({\rm Def}(X)) \neq \emptyset$.
\end{problem}


Let us translate this problem into the terminology 
of the Weyl group of $GL_h$.
Let $W = W_h$ be the Weyl group of $GL_h$.
We identify this $W$ with the symmetric group $\mathfrak S_h$ in the usual way.
We define $J = J_c$ by $J_c = \{s_1, \dots, s_h\} - \{s_c\}$,
where $s_i$ is the simple reflection $(i, i+1)$.
Put $d = h - c$.
Then there exists a one-to-one correspondence
between the isomorphism classes of ${\rm BT_1}$'s 
of rank $p^h$ and dimension $d$ over $k$ 
and the subset ${}^J W$ of $W$;
see Section~\ref{ClassOfBT1}.
Let $X$ be a $p$-divisible group.
Let $w \in {}^J W$.
We say $w$ is the ($p$-kernel) type of $X[p]$
if the ${\rm BT_1}$ $X[p]$ corresponds to $w$ by this bijection.

In Proposition~\ref{PropOfDefXDefY} we will show that:
Let $X$ and $Y$ be $p$-divisible groups over $k$
with $\mathcal C_{H(\xi)}({\rm Def}(X)) \neq \emptyset$
and $X[p] \simeq Y[p]$.
Then $\mathcal C_{H(\xi)}({\rm Def}(Y)) \neq \emptyset$.
Thanks to this proposition, Problem~\ref{ProblemOfCHxiDefXneqEmpty}
is reduced to
\begin{problem} \label{PrpblemOfwinJW}
Classify elements $w$ of ${}^J W$ such that
\begin{itemize}
\item[($\ast$)] there exists a $p$-divisible group $X$ over $k$ such that
$w$ is the type of $X[p]$ satisfying that 
$\mathcal C_{H(\xi)}({\rm Def}(X)) \neq \emptyset$.
\end{itemize}
\end{problem}
In this paper, we treat the following problem:
\begin{problem} \label{ProblemOfLength-1}
Classify $w \in {}^J W$ satisfying that $(\ast)$ and 
$\ell(w) = \ell(H(\xi)[p]) - 1$.
\end{problem}
A complete answer to this problem will be given by combining
Theorem~\ref{ThmOfOneToOne} and Theorem~\ref{ThmOfSpeNP} below.
In \cite{HiguchiBC}, 
we solved Problem~\ref{ProblemOfLength-1}
for Newton polygons $\xi$ consisting of two slopes satisfying
that one slope is less than 1/2 and the
other slope is greater than 1/2.
Our proof reduces the problem to the case of \cite{HiguchiBC}.

Before we state the main theorems,
we explain the above formulations in terms of 
specializations of $p$-divisible groups.
Let $X$ and $Y$ be $p$-divisible groups over $k$.
We say $X$ is a {\it specialization} of $Y$
if there exists a family of $p$-divisible group $\mathfrak X \to {\rm Spec}(R)$
with discrete valuation ring $R$ in characteristic of $p$ such that
$\mathfrak X$ is isomorphic to $Y$ over an algebraically closed field
containing $L$ and $k$,
and $\mathfrak X_k$ is isomorphic to $X$ over an algebraically closed field
containing $K$ and $k$,
where $L$ is the field of fractions of $R$, 
and $K = R/\mathfrak m$ is the residue field of $R$.
Note that $X$ is a specialization of $Y$ if and only if
$\mathcal C_Y({\rm Def}(X)) \neq \emptyset$ holds.
For a $p$-divisible group $X$,
we define the length $\ell(X[p])$ of the $p$-kernel by
the length of the element of the Weyl group which is the type of $X[p]$.
It is known that for the $p$-divisible group $X_0$,
the length $\ell(X_0[p])$ is equal to the dimension of the locally
closed subscheme of ${\rm Def}(X_0)$ obtained by
giving the induced reduced structure to the subset of ${\rm Def}(X_0)$
consisting of points $s \in {\rm Def}(X_0)$ such that
$\mathfrak X'_s[p]$ is isomorphic to $X_0[p]$
over an algebraically closed field;
see \cite[6.10]{wedhorn} and \cite[3.1.6]{moonen}.
We say a specialization $X$ of $Y$ is {\it generic} if
$\ell(X[p]) = \ell(Y[p]) - 1$ holds.


For an arbitrary Newton polygon $\xi$,
we denote by $B(\xi)$ the set
\begin{eqnarray} \label{EqOfBxi}
B(\xi) = \{\text{types of } X_s[p] \mid X_{\overline \eta} = H(\xi) \text{ and } 
\ell(X_s[p]) = \ell(X_{\overline \eta}[p]) - 1 \text{ for some } X \to {\rm S}\},
\end{eqnarray}
where
${\rm S} = {\rm Spec}(R)$ for a discrete valuation ring $(R, \mathfrak m)$,
$s = {\rm Spec}(\kappa)$ and $\overline \eta = {\rm Spec}(\overline K)$
with $\kappa = R/\mathfrak m$ and $K = {\rm frac}(R)$.
Problem~\ref{ProblemOfLength-1} asks us to
determine the set $B(\xi)$.
The first result is:
\begin{theorem} \label{ThmOfOneToOne}
Let $\xi = \sum_{i=1}^z (m_i, n_i)$ be a Newton polygon.
Let $\xi_i = (m_i, n_i) + (m_{i+1}, n_{i+1})$ be the Newton polygon
consisting of two adjacent segments for $i = 1, \dots z$.
The map 
\begin{eqnarray}
\bigsqcup_{i=1}^{z-1} B(\xi_i) \to B(\xi)
\end{eqnarray} which sends
an element $w$ of $B(\xi_i)$ to $w_{\zeta_i} \oplus w$,
where $w_{\zeta_i}$ is the type of $H(\zeta_i)[p]$ with $\zeta_i 
= (m_1, n_1) + \cdots + (m_{i-1}, n_{i-1}) + (m_{i+1}, n_{i+1}) + \cdots + (m_z, n_z)$
is bijective.
\end{theorem}
This theorem implies that the determining problem of boundaries of central streams
is reduced to the case that the Newton polygon consists of two segments.
Moreover, for the two slopes case, we will show the following result:
\begin{theorem} \label{ThmOfSpeNP}
Let $\xi = (m_1, n_1) + (m_2, n_2)$ be a Newton polygon satisfying that
$n_1/(m_1 + n_1) > n_2/(m_2 + n_2) \geq 1/2$.
Put $\xi^{\rm C} = (m_1, n_1- m_1) + (m_2, n_2 - m_2)$.
Then the map sending $w$ to $w|_{\{1, \dots, n_1+n_2\}}$
gives a bijection from $B(\xi)$ to $B(\xi^{\rm C})$.
\end{theorem}
For a Newton polygon $\xi = (m_1, n_1) + (m_2, n_2)$,
we set $\xi^{\rm D} = (n_2, m_2) + (n_1, m_1)$.
By the duality, it is easy to see that
the map sending $w$ to $i \mapsto l - w(l - i)$,
with $l = m_1 + n_1 + m_2 + n_2 + 1$,
gives a bijection from $B(\xi)$ to $B(\xi^{\rm D})$.
Using repeatedly this duality and Theorem~\ref{ThmOfSpeNP},
we can reduce Problem~\ref{ProblemOfLength-1}
to the case of \cite{HiguchiBC}.
There results give a complete answer to Problem~\ref{ProblemOfLength-1}.

This paper is organized as follows.
In Section~\ref{Prelim}, we recall notions of 
$p$-divisible groups, Newton polygons
and truncated Dieudonn\'e modules of level one.
In this section, we review the classification of ${\rm BT_1}$'s.
Moreover, we introduce the definition of 
arrowed binary sequences which is the 
main tool to show the main result.
In Section~\ref{ArbitNP}, we show some properties of arrowed binary sequences
and Newton polygons.
The goal of this section is to give a proof of Theorem~\ref{ThmOfOneToOne}.
In Section~\ref{TwoSlopesNP},
we treat central streams 
corresponding to Newton polygons consisting of two slopes,
and give a proof of Theorem~\ref{ThmOfSpeNP}.

\section{Preliminaries} \label{Prelim}

In this section, first we recall the notions of $p$-divisible groups,
leaves and Dieudonn\'e modules.
In Section~\ref{ClassOfBT1},
we review the definition of truncated Barsotti-Tate groups of level one
and classification of ${\rm BT_1}$'s.
Moreover, in Section~\ref{SecABS} we introduce arrowed binary sequences
as a generalization of classifying data ${}^J W$ of ${\rm BT_1}$'s,
which are the main tool to show the main theorem.

\subsection{$p$-divisible groups and Dieudonn\'e modules} \label{Definitions}

In this section we fix a prime number $p$.
Let $h$ be a non-negative integer.
Let ${\rm S}$ be a scheme in characteristic $p$.
A \textit{$p$-divisible group} (Barsotti-Tate group) of height $h$
over ${\rm S}$ is an inductive system $(G_v, i_v)_{v \geq 1}$,
where $G_v$ is a finite locally free commutative group scheme over ${\rm S}$
of order $p^{vh}$, 
and for every $v$, there exists the exact sequence of commutative group schemes
$$0 \longrightarrow G_v \overset{i_v}\longrightarrow G_{v+1}
\overset{p^v}\longrightarrow G_{v+1},$$
with canonical inclusion $i_v$.
Let $X = (G_v, i_v)$ be a $p$-divisible group over ${\rm S}$.
Let $T$ be a scheme over ${\rm S}$.
Then we have the $p$-divisible group $X_T$ over $T$,
which is defined as $(G_v \times_{\rm S} T,\ i_v \times {\rm id})$.
For the case $T$ is a closed point $s$ over ${\rm S}$,
we call the $p$-divisible group $X_s$ fiber of $X$ over $s$.
Let $k$ be an algebraically closed field of characteristic $p$.
Let $Y \to {\rm Spec}(k)$ a $p$-divisible group,
and let $\mathcal Y \to {\rm S}$ be a $p$-divisible group.
In \cite[2.1]{oortFoliations} Oort defined a \textit{leaf} by

\begin{equation} \label{EqOfLeaf}
\mathcal C_Y({\rm S}) = 
\{s \in {\rm S} \mid \mathcal Y_s \text{ is isomorphic to } Y
\text{ over an algebraically closed field}\},
\end{equation}
as a set,
and showed that $\mathcal C_Y(\rm S)$ is closed
in a Newton stratum (cf. \cite[2.2]{oortFoliations}).
We regard $\mathcal C_Y({\rm S})$ as a locally closed subscheme of ${\rm S}$
by giving the induced reduced structure on it.

Let $K$ be a perfect field of characteristic $p$.
Let $W(K)$ denote the ring of Witt-vectors with coefficients in $K$.
Let $\sigma$ be the Frobenius over $K$.
We denote by the same symbol $\sigma$ the Frobenius over $W(K)$
if no confusion can occur.
A {\it Dieudonn\'e module over $K$} is a finite $W(K)$-module $M$
equipped with $\sigma$-linear homomorphism ${\rm F} : M \to M$
and $\sigma^{-1}$-linear homomorphism ${\rm V} : M \to M$
satisfying that ${\rm F} \circ {\rm V}$ and ${\rm V} \circ {\rm F}$ 
equal the multiplication by $p$.
For each $p$-divisible group $X$, we have the Dieudonn\'e module 
$\mathbb D(X)$ using the covariant Dieudonn\'e functor.
The covariant Dieudonn\'e theory says that
the functor $\mathbb D$ induces a canonical categorical equivalence between
the category of $p$-divisible groups over $K$ and
that of Dieudonn\'e modules over $K$ which are free as $W(K)$-modules.
In particular, there exists a categorical equivalence from
the category of finite commutative group schemes over $K$ to
that of Dieudonn\'e modules over $K$ which are of finite length.

Let $\{(m_i, n_i)\}_i$ be finite number of pairs of coprime non-negative integers 
satisfying that $\lambda_i \geq \lambda_j$ for $i < j$,
where $\lambda_i = n_i/(m_i + n_i)$ for each $i$.
A {\it Newton polygon} is a 
lower convex polygon in $\mathbb R^2$,
which breaks on integral coordinates and 
consists of slopes $\lambda_i$.
We write
\begin{equation} \label{EqOfNP}
\sum_i (m_i, n_i)
\end{equation}
for the Newton polygon.
We call each coprime pair $(m_i, n_i)$ {\it segment}.
For a Newton polygon $\xi = \sum_i (m_i, n_i)$,
we define the $p$-divisible group $H(\xi)$ by
\begin{equation} \label{EqOfMinpdiv}
H(\xi) = \bigoplus_i H_{m_i, n_i},
\end{equation}
where $H_{m, n}$ is the $p$-divisible group over $\mathbb F_p$
which is of dimension $n$, and its Serre-dual is of dimension $m$.
Moreover the Dieudonn\'e module $\mathbb D(H_{m, n})$ satisfies that
\begin{equation} \label{EqOfDieudonne}
\mathbb D(H_{m, n}) = \bigoplus_{i=1}^{m+n} W(\mathbb F_p) e_i,
\end{equation}
with $W(\mathbb F_p)$ is the ring of Witt-vectors over $\mathbb F_p$,
and $e_i$ is a basis.
In this case $W(\mathbb F_p)$ is equal to the ring of $p$-adic integers $\mathbb Z_p$.
For the basis $e_i$, operations ${\rm F}$ and ${\rm V}$ satisfy that
${\rm F}e_i = e_{i - m}$, ${\rm V}e_i = e_{i - n}$ and $e_{i - (m+n)} = pe_i$.

We say a $p$-divisible group $X$ is {\it minimal} if
$X$ is isomorphic to $H(\xi)$ over an algebraically closed field 
for a Newton polygon $\xi$.
For a $p$-divisible group $X$, the $p$-kernel $X[p]$ is obtained by
the kernel of the multiplication by $p$.
It is known that the Dieudonn\'e module of $H(\xi)[p]$ makes 
a truncated Dieudonn\'e module of level one
(abbreviated as ${\rm DM_1}$) $\mathbb D(H(\xi)[p])$.
A ${\rm DM_1}$
over $K$ of height $h$ 
is the triple $(N, {\rm F}, {\rm V})$
consisting of a $K$-vector space $N$ of height $h$,
a $\sigma$-linear map and a $\sigma^{-1}$-linear map
from $N$ to itself satisfying that
${\rm ker}\, {\rm F} = {\rm im}\, {\rm V}$ and 
${\rm im}\, {\rm F} = {\rm ker}\, {\rm V}$.

Let $\xi = \sum (m_i, n_i)$ be a Newton polygon.
We denote by $N_\xi$ the ${\rm DM_1}$ associated to the 
$p$-kernel of $H(\xi)$.
Then $N_\xi$ is described as
\begin{equation}
N_\xi = \bigoplus N_{m_i, n_i},
\end{equation}
where $N_{m, n}$ is the ${\rm DM_1}$ corresponding to the $p$-kernel of $H_{m, n}$.

We use the same notation as Section~\ref{Intro}.
The following proposition would be well-known to the specialists,
but as any good reference cannot be found,
we have give a proof for the reader's convenience.
A proof in the polarized case is given in \cite[12.5]{oortstr}.
\begin{proposition} \label{PropOfDefXDefY}
Let $\xi$ be a Newton polygon.
Let $X$ and $X'$ be $p$-divisible groups over an algebraically closed field
of characteristic $p$.
If $\mathcal C_Y({\rm Def}(X)) \neq \emptyset$
and $X[p] \simeq X'[p]$, 
then $\mathcal C_Y({\rm Def}(X')) \neq \emptyset$.
\end{proposition}

\begin{proof}
Let  $h$ and $c$ be positive integers such that
$X[p]$ is the type of $w \in {}^J W$ with $W = W_h$ and $J = J_c$.
Put $d = h - c$.
Let ${\rm F}$ (resp. ${\rm V}$) denote the $\sigma$-linear map
(resp. $\sigma^{-1}$-linear map) of 
the ${\rm DM_1}$ $\mathbb D(X[p]) = \mathbb D(X)/p\mathbb D(X)$
with $\sigma$ the Frobenius.
We take a basis $\bar{z}_{d+1}, \dots, \bar{z}_h$ of 
the image of ${\rm V}$,
and choose $\bar{z}_1, \dots, \bar{z}_d \in \mathbb D(X[p])$ to be
$\bar{z}_1, \dots, \bar{z}_h$ is a basis of $\mathbb D(X[p])$.
Let $z_1, \dots, z_h$ denote the lift of $\bar{z}_1, \dots, \bar{z}_h$ to $\mathbb D(X)$.
Then $\{z_1, \dots, z_h\}$ is a basis of $\mathbb D(X)$.
We write
$$
\begin{pmatrix}
A & B \\
C & D \\
\end{pmatrix}
$$
for the display of $X$ with respect to the basis $z_1, \dots, z_h$,
where $A$ is the $d \times d$ matrix, and
$D$ is the $(h-d) \times (h-d)$ matrix.
See \cite{NormanOortModuliOfAV} for the construction of the display.
Then for the Dieudonn\'e module $\mathbb D(X)$ of $X$
equipped with the operations ${\rm F}$ and ${\rm V}$,
we have
$$({\rm F} z_1, \dots, {\rm F} z_h) = (z_1, \dots, z_h) 
\begin{pmatrix}
A & pB \\
C & pD \\
\end{pmatrix}$$
and
$$({\rm V} z_1, \dots, {\rm V} z_h) = (z_1, \dots, z_h) 
\begin{pmatrix}
p\alpha & p\beta \\
\gamma & \delta \\
\end{pmatrix}^{\sigma^{-1}},$$
where
$$
\begin{pmatrix}
\alpha & \beta \\
\gamma & \delta \\
\end{pmatrix}
$$
is the inverse matrix of the display of $X$.
The operations ${\rm F}$ and ${\rm V}$ on 
$\mathbb D(X[p])$ 
satisfy that
$$({\rm F} \bar{z}_1, \dots, {\rm F} \bar{z}_h) = (\bar{z}_1, \dots, \bar{z}_h) 
\begin{pmatrix}
\bar{A} & 0 \\
\bar{C} & 0 \\
\end{pmatrix}$$
and
$$({\rm V} \bar{z}_1, \dots, {\rm V} \bar{z}_h) = (\bar{z}_1, \dots, \bar{z}_h) 
\begin{pmatrix}
0 & 0 \\
\bar\gamma & \bar\delta \\
\end{pmatrix}^{\sigma^{-1}}.$$
For the $p$-divisible group $\mathfrak X \to {\rm Spec}(R)$ 
corresponding to the universal $p$-divisible group over ${\rm Spf}(R)$,
the display of $\mathfrak X$ induces that
$$({\rm F} \bar{z}_1, \dots, {\rm F} \bar{z}_h) = (\bar{z}_1, \dots, \bar{z}_h) 
\begin{pmatrix}
\bar{A} + \bar{T} \bar{C} & 0 \\
\bar{C} & 0 \\
\end{pmatrix}$$
and
$$({\rm V} \bar{z}_1, \dots, {\rm V} \bar{z}_h) = (\bar{z}_1, \dots, \bar{z}_h) 
\begin{pmatrix}
0 & 0 \\
\bar\gamma & -\bar\gamma \bar{T} + \bar\delta \\
\end{pmatrix}^{\sigma^{-1}},$$
where $\bar{T}$ is an $(h-n) \times n$ matrix on $R$.
On the other hand, let 
$$
\begin{pmatrix}
a & b \\
c & d \\
\end{pmatrix}
$$
denote the display of $X'$.
By the isomorphism from $X[p]$ to $X'[p]$,
we have the basis $\bar{e}_1, \dots, \bar{e}_h$ of 
$\mathbb D(X'[p]) = \mathbb D(X')/p \mathbb D(X')$.
We have then
$$({\rm F} \bar{z}_1, \dots, {\rm F} \bar{z}_h) = (\bar{z}_1, \dots, \bar{z}_h) 
\begin{pmatrix}
\bar{a} & 0 \\
\bar{c} & 0 \\
\end{pmatrix}$$
and
$$({\rm V} \bar{z}_1, \dots, {\rm V} \bar{z}_h) = (\bar{z}_1, \dots, \bar{z}_h) 
\begin{pmatrix}
0 & 0 \\
\bar{\gamma'} & \bar{\delta'} \\
\end{pmatrix}^{\sigma^{-1}}$$
for the inverse matrix
$$
\begin{pmatrix}
\alpha' & \beta' \\
\gamma' & \delta' \\
\end{pmatrix}
$$
of the display of $X'$.
Let $\mathcal Y$ be the $p$-divisible group having
$$
\begin{pmatrix}
1 & T \\
0 & 1 \\
\end{pmatrix}
\begin{pmatrix}
a & b \\
c & d \\
\end{pmatrix}
$$
as its display, where $T$ is a matrix with $T \bmod p = \bar{T}$.
Then for the display of $\mathcal Y[p]$, we see
$$({\rm F} \bar{z}_1, \dots, {\rm F} \bar{z}_h) = (\bar{z}_1, \dots, \bar{z}_h) 
\begin{pmatrix}
\bar{a} + \bar{T} \bar{c} & 0 \\
\bar{c} & 0 \\
\end{pmatrix}$$
and
$$({\rm V} \bar{z}_1, \dots, {\rm V} \bar{z}_h) = (\bar{z}_1, \dots, \bar{z}_h) 
\begin{pmatrix}
0 & 0 \\
\bar{\gamma'} & -\bar{\gamma'} \bar{T} + \bar{\delta'} \\
\end{pmatrix}^{\sigma^{-1}},$$
whence $\mathcal Y$ belongs to $\mathcal C_Y({\rm Def}(X'))$.
\end{proof}

\subsection{Classification of ${\rm BT_1}$'s} \label{ClassOfBT1}

In this section, we work over an algebraically closed field $k$.
Let us review the classification of 
truncated Barsotti-Tate groups of level one.
\begin{definition} \label{DefOfBT1}
A {\it truncated Barsotti-Tate group of level one} 
(${\rm BT_1}$) is a commutative,
finite and flat group scheme $N$ over a scheme in characteristic $p$
satisfying properties $[p]_N = 0$,
and 
\begin{eqnarray}
{\rm im}\, ({\rm V}: N^{(p)} \to N) &=& {\rm ker}\, ({\rm F} : N \to N^{(p)}),\\
{\rm im}\, ({\rm F} : N \to N^{(p)}) &=& {\rm ker}\, ({\rm V}: N^{(p)} \to N).
\end{eqnarray}
\end{definition}
A ${\rm DM_1}$ appears as a Dieudonn\'e module of
a ${\rm BT_1}$.
Let $W = W_h$ be the Weyl group of the general linear group $GL_h$.
This $W$ can be identified with the symmetric group $\mathfrak S_h$.
Let $\Omega$ be the standard generator of $W = \mathfrak S_h$.
We denote by $s_i$ the simple reflection $(i, i+1)$.
We define $J = J_c$ by $J_c = \Omega - \{s_c\}$.
Put $d = h - c$.
For the set $W_J := W_c \times W_d$,
let ${}^J W$ be the set consisting of elements $w$ of $W_h$
such that $w$ is the shortest element of ${}^J W \cdot w$;
see \cite[Chap. IV, Ex. \S 1 (3)]{BourbakiLieGps}.
Then we have 
\begin{theorem} \label{ThmOfJWBT1}
There exists a one-to-one correspondence
\begin{eqnarray}
{}^J W \longleftrightarrow 
\{{\rm BT_1}\text{'s} \text{ over } k \text{ of height } h \text{ of dimension } d\} / \cong.
\end{eqnarray}
Moreover, running over all $d$, we have
\begin{eqnarray}
\bigsqcup_d {}^J W \longleftrightarrow \{0, 1\}^h.
\end{eqnarray}
\end{theorem}

Kraft \cite{kraftKom}, Oort \cite{oortstr} 
and Moonen-Wedhorn \cite{moonen-wedhornDis} show that
there exists a one-to-one correspondence:
\begin{eqnarray}
\{0, 1\}^h \longleftrightarrow \{{\rm DM_1}\text{'s} \text{ over } k 
\text{ of height } h \} / \cong.
\end{eqnarray}
For $\nu \in \{0, 1\}^h$,
we construct the ${\rm DM_1}$ $D(\nu)$ as follows.
We write $\nu(i)$ for the $i$-th coordinate of $\nu$.
Put $N = ke_1\oplus \cdots \oplus ke_h$.
We define maps ${\rm F}$ and ${\rm V}$ as follows:
\begin{eqnarray} \label{DefOfF}
{\rm F}e_i =
\begin{cases}
e_j,\ j = \#\{ l \mid \nu(l) = 0,\ l \leq i\} & {\rm for}\ \nu(i) = 0,\\
0 & {\rm otherwise}.
\end{cases}
\end{eqnarray}
Let $j_1, \dots, j_c$,
with $j_1 < \dots < j_c$, be
the natural numbers satisfying $\nu(j_l) = 1$.
Put $d = h - c$.
Then a map ${\rm V}$ is defined by
\begin{eqnarray} \label{DefOfV}
{\rm V}e_i =
\begin{cases}
e_{j_l},\ l = i - d & {\rm for}\ i > d,\\
0 & {\rm otherwise}.
\end{cases}
\end{eqnarray}
Therefore $D(\nu)$ is given by $D(\nu) = (N, {\rm F}, {\rm V})$.
Thus we can identify ${\rm DM_1}$'s with sequences consisting of $0$ and $1$.

For $w \in {}^J W$,
we define $\nu(j) = 0$ if and only if $w(j) > c$ for $j = 1, \dots, h$,
and we obtain the element $(\nu(1), \nu(2), \dots, \nu(h))$
of $\{0, 1\}^h$.
This gives a one-to-one correspondence between
${}^J W$ and the subset of $\{0, 1\}^h$ consisting of elements $\nu$
satisfying $\# \{j \mid \nu(j) = 0\} = d$.
Thus we obtain a bijection between ${}^J W$ and the set of isomorphism class
of ${\rm DM_1}$'s over $k$ of height $h$ and dimension $d$.

We say that $w'$ is a specialization of $w$,
denoted by $w' \subset w$, if 
there exists a discrete valuation ring $R$ of characteristic $p$ such that
there exists a finite flat commutative group scheme $G$ over $R$
satisfying that $G_{\overline \kappa}$ is a ${\rm BT_1}$ of the type $w'$,
and $G_{\overline L}$ is a ${\rm BT_1}$ of the type $w$,
where $L$ (resp. $\kappa$) is the fractional field of $R$ 
(resp. is the residue field of $R$). 
A {\it generic specialization} $w'$ of $w$ is a specialization of $w$
satisfying $\ell(w') = \ell(w) - 1$.

Here, we show a lemma used for the construction of generic specializations.
We define $x \in W$ by $x(i) = i + d$ if $i \leq c$ and $x(i) = i - c$ otherwise.
Let $\theta$ be the map from $W$ to itself defined by
$\theta(u) = xux^{-1}$.
By \cite[4.10]{VWEOShimura},
we have $w' \subset w$ if and only if there exists $u \in W_J$
such that $u^{-1} w' \theta(u) \leq w$
with the Bruhat order $\leq$.
Let us recall \cite[Lemma~2.7]{HiguchiBC}.

\begin{lemma} \label{LemOfw'w}
Let $w \in {}^J W$.
Let $w'$ be a specialization of $w$.
If $w'$ is generic,
then there exist $v \in W$ and $u \in W_J$ such that
\begin{enumerate}
\item[(i)] $v = ws$ for a transposition $s$,
\item[(ii)] $\ell(v) = \ell(w) - 1$,
\item[(iii)] $w' = uv\theta(u^{-1})$.
\end{enumerate}
\end{lemma}

\subsection{Arrowed binary sequences} \label{SecABS}

In order to classify the types of generic specializations of $H(\xi)$,
we introduce a new notion {\it arrowed binary sequence}
which is slightly more generalized than the notion of binary sequences $\{0, 1\}^h$.
This notion of arrowed binary sequences is useful for
classifying generic specializations of minimal $p$-divisible groups.

\begin{definition} \label{DefOfABS}
An {\it arrowed binary sequence} 
(we often abbreviate as ABS)
$S$ is the triple $(T, \Delta, \Pi)$ consisting of 
an ordered symbol set $T = \{t_1 < t_2 < \cdots < t_h\}$,
a map $\Delta : T \to \{0, 1\}$ and
a bijection $\Pi : T \to T$.
For an ABS $S$, let $T(S)$ denote the ordered symbol set of $S$.
Similarly, we denote by $\Delta(S)$ (resp. $\Pi(S)$) 
the map from $T(S)$ to $\{0, 1\}$ (resp. the map from $T(S)$ to itself).
For an ABS $S$,
we define the {\it length} $\ell(S)$ of $S$ by
\begin{eqnarray}
\ell(S) = \# \{(t, t') \in T(S) \times T(S) \mid 
t < t' \text{ with } \delta(t) = 0 \text{ and } \delta(t') = 1\}.
\end{eqnarray}
\end{definition}

\begin{remark}
Let $N = (N, {\rm F}, {\rm V})$ be a ${\rm DM_1}$.
We construct the arrowed binary sequence $(\Lambda, \delta, \pi)$ 
associated to $N$ as follows.
Let $\nu$ be the element of $\{0, 1\}^h$ corresponding to $N$. 
For an totally ordered set $\Lambda = \{t_1, \dots, t_h\}$,
let $\delta:\Lambda \rightarrow \{0, 1\}$ be the map
which sends $t_i$ to the $i$-th coordinate of $\nu$.
We define a map $\pi:\Lambda \rightarrow \Lambda$ by
$\pi(t_i) = t_j$,
where $j$ is uniquely determined by
\begin{equation}
\begin{cases}
{\rm F}e_i = e_j & \text{ if } \delta(t_i) = 0,\\
{\rm V}e_j = e_i & \text{ otherwise}.
\end{cases}
\end{equation}
We say an ABS $S$ is {\it admissible}
if there exists a ${\rm DM_1}$ such that
$S$ is obtained from this ${\rm DM_1}$ as above.
\end{remark}

\begin{remark} \label{RemOfSimpleABS}
For the ${\rm DM_1}$ $N_{m, n}$ corresponding to 
the $p$-divisible group $H_{m, n}$,
we get the ABS $S$ as follows.
Set $T(S) = \{t_1, \dots, t_{m+n}\}$.
The map $\Delta(S)$ is defined by
$\Delta(S)(t_i) = 1$ if $i \leq m$,
and $\Delta(S)(t_i) = 0$ otherwise. 
The map $\Pi(S)$ is defined by
$\Pi(S)(t_i) = t_{i - m \bmod (m+n)}$.
\end{remark}

Let $S$ be an ABS.
Put $\delta = \Delta(S)$ and $\pi = \Pi(S)$.
The {\it binary expansion} $b(t)$ of $t \in T(S)$ is the real number
$b(t) = 0.b_1b_2.\dots,$ 
where $b_i = \delta(\pi^{-i}(t))$.

\begin{proposition} \label{PropOfBinExp}
Let $S$ be an admissible ABS.
For elements $t_i$ and $t_j$ of $T(S) = \{t_1, t_2, \dots, t_h\}$, 
the following holds.
\begin{itemize}
\item[(i)] Suppose $\Delta(S)(t_i) = \Delta(S)(t_j)$.
Then $t_i < t_j$ if and only if $\Pi(S)(t_i) < \Pi(S)(t_j)$.
\item[(ii)] Suppose $b(t_i) \neq b(t_j)$.
Then $b(t_i) < b(t_j)$ if and only if $i < j$.
\end{itemize}
\end{proposition}

\begin{proof}
(i) follows from the construction of admissible ABS's.
Let us see (ii).
Put $\delta = \Delta(S)$ and $\pi = \Pi(S)$.
By the construction of admissible ABS's,
for elements $t$ and $t'$ of $T(S)$,
if $\delta(t) = 1$ and $\delta(t') = 0$,
then $\pi(t') < \pi(t)$.
First, assume $b(t_i) < b(t_j)$.
Then there exists a non-negative integer $u$ such that
$\delta(\pi^{-v}(t_i)) = \delta(\pi^{-v}(t_j))$ for $0 \leq v < u$
and $\delta(\pi^{-u}(t_i)) = 0$, $\delta(\pi^{-u}(t_j)) = 1$.
We have then $\pi^{-u+1}(t_i) < \pi^{-u+1}(t_j)$,
and the assertion follows from (i).
Next, assume $i < j$.
To lead a contradiction, we suppose that $b(t_j) < b(t_i)$.
Then there exists a non-negative integer $u$ such that
$\delta(\pi^{-u}(t_j) = 0$ and $\delta(\pi^{-u}(t_i)) = 1$,
and for non-negative integers $v$ satisfying $v < u$,
we have $\delta(\pi^{-v}(t_j) = \delta(\pi^{-v}(t_i))$.
This implies that $\pi^{-u+1}(t_j) < \pi^{-u+1}(t_i)$,
and we have a contradiction.
\end{proof}

We denote by $\mathcal H'(h, d)$ the set of admissible ABS's
whose corresponding ${\rm DM_1}$'s are of height $h$ and dimension $d$.
We shall translate the ordering $\subset$ on ${}^J W$ via
the bijection from ${}^J W$ to $\mathcal H'(h, d)$,
and we obtain an ordering on $\mathcal H'(h, d)$ as 
the notion of specializations of admissible ABS's.

We shall give a method to construct a type of the specializations of ABS's.
It will turns out to correspond to specializations $w' \subset w$ with
$v = ws < w$ and $w' = uv\theta(u^{-1})$,
where $s$ denotes a transposition and $u \in W_J$.
From $S \in \mathcal H'(h, d)$,
we construct a new admissible ABS $S'$.

\begin{definition} \label{DefOfSmallModification}
Let $S$ be an ABS with $T(S) = \{t_1 < \dots < t_h\}$.
Let $i$ and $j$ be natural numbers satisfying that
$\Delta(S)(t_i) = 0$ and $\Delta(S)(t_j) = 1$ with $i < j$.
We define an ABS $S^{(0)}$ as follows.
We set $T(S^{(0)}) = \{t'_1 <' \dots <' t'_h\}$
to be $t'_z = t_{f(z)}$ for $f = (i, j)$ transposition.
Let $\Delta(S^{(0)}) = \Delta(S)$.
For a natural number $z$ with $1 \leq z \leq h$,
we denote by $g(z)$ the natural number satisfying $\Pi(S)(t_z) = t_{g(z)}$.
We define $\Pi(S^{(0)}) : T(S^{(0)}) \to T(S^{(0)})$ by
\begin{eqnarray}
\Pi(S^{(0)})(t'_z) = 
\begin{cases}
t'_{g(j)} & \text{if } z = i, \\
t'_{g(i)} & \text{if } z = j, \\
t'_{g(z)} & \text{otherwise}.
\end{cases}
\end{eqnarray}
Thus we obtain an ABS $S^{(0)}$.
We call this ABS {\it small modification by $t_i$ and $t_j$}.
\end{definition}

\begin{definition} \label{DefOfSpecialization}
Let $S$ be an ABS.
Let $S^{(0)}$ be the small modification by $t_i$ and $t_j$.
Put $T(S^{(0)}) = \{t_1 < \dots < t_h\}$.
We define an ABS $S'$ as follows.
Let $T(S') = T(S^{(0)})$ as sets.
Let $<'$ denote an ordering of $T(S')$.
Put $\Delta(S') = \Delta(S^{(0)})$ and $\Pi(S') = \Pi(S^{(0)})$.
We say that 
the ABS $S'$ is a {\it specialization of $S$}
if for elements $t_x$ and $t_y$ of $T(S')$,
$$t_x <' t_y \Rightarrow b(t_x) \leq b(t_y).$$
We say a specialization $S'$ of $S$ is {\it generic} if
$\ell(S') = \ell(S) - 1$.
\end{definition}

Note that for the small modification $S^{(0)}$,
in general the specialization $S'$ is not unique.
However, the ${\rm DM_1}$ obtained by the pair $(T(S'), \Delta(S'))$
is unique.

\begin{remark} \label{RemOfw'w}
Let $S \in \mathcal H'(h, d)$.
Let $S'$ be the specialization of $S$ obtained by exchanging $t_i$ and $t_j$
with $T(S') = \{t'_1 <' \dots <' t'_h\}$.
We denote by $w$ the element of ${}^J W$ corresponding to $S$.
Put $s = (i, j)$ transposition.
Maps $\Pi(S)$ and $\Pi(S')$ can be regarded as elements of $W$.
We have then $\Pi(S) = xw$.
For the small modification $S^{(0)}$
with $T(S^{(0)}) = \{t^{(0)}_1 < \dots < t^{(0)}_h\}$,
we define $\varepsilon \in W$ to be
$t^{(0)}_z = t'_{\varepsilon(z)}$.
Since $b(t^{(0)}_z) < 0.1$ if $z \leq d$ and $b(t^{(0)}_z) > 0.1$ otherwise,
$\varepsilon$ stabilizes $\{1, 2, \dots, d\}$.
Put $v = ws$.
Then $w' = uv\theta(u^{-1})$ corresponds to $S'$
for $u = x^{-1} \varepsilon^{-1} x \in W_J$.
The map $\Pi(S')$ is obtained by $\varepsilon^{-1} \Pi(S) s \varepsilon$.
\end{remark}

Next, in Definition~\ref{DefOfDirectSum}, 
we introduce the direct sum of ABS's.
The construction of the direct sum is induced from the direct sum of 
corresponding ${\rm DM_1}$'s.

\begin{definition} \label{DefOfDirectSum}
Let $S_1$ and $S_2$ be ABS's.
We define the {\it direct sum} $S = S_1 \oplus S_2$ 
of $S_1$ and $S_2$ as follows.
Let $T(S) = T(S_1) \sqcup T(S_2)$ as sets.
We define the map $\Delta(S) : T(S) \to \{0, 1\}$ to be 
$\Delta(S)|_{T(S_i)} = \Delta(S_i)$ for $i = 1, 2$.
Let $\Pi(S)$ be the map from $T(S)$ to itself
satisfying that $\Pi(S)|_{T(S_i)} = \Pi(S_i)$ for $i = 1, 2$.
We define the order on $T(S)$ so that
for elements $t$ and $t'$ of $T(S)$,
\begin{itemize}
\item[(i)] if $b(t) \leq b(t')$, then $t < t'$;
\item[(ii)] $t < t'$ if and only if $\Pi(S)(t) < \Pi(S)(t')$
when $\Delta(S)(t) = \Delta(S)(t')$.
\end{itemize}
\end{definition}

\begin{notation} \label{NotOfNxi}
Let $N_\xi$ be the minimal ${\rm DM_1}$ 
of a Newton polygon $\xi = \sum_{i=1}^z (m_i, n_i)$. 
Let $S$ be the ABS associated to $N_\xi$.
Then $S$ is described as $S = \bigoplus_{i=1}^z S_i$,
where $S_i$ is the ABS associated to the ${\rm DM_1}$ $N_{m_i, n_i}$.
If an element $t$ of $T(S)$ belongs to $T(S_r)$,
then we denote by $t^r$ or $\tau^r$ this element $t$
with $\tau = \Delta(S)(t)$.
If we want to say that the element $t^r$ is the $i$-th element of $T(S_r)$,
we write $t^r_i$ for the element $t^r$.
Furthermore,
we often write $\tau^r_i$ for the element $t^r_i$ of $T(S)$ 
with $\tau = \Delta(S)(t^r_i)$.
Moreover, we describe the map $\Pi(S)$ by arrows:
$$
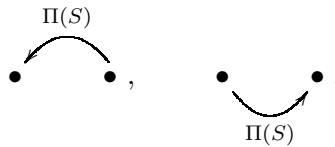
\xymatrix{\bullet & \bullet \ar@/_15pt/[l]_{\Pi(S)} \ ,
& \bullet \ar@/_15pt/[r]_{\Pi(S)} &\bullet}.$$
\end{notation}

\begin{example} \label{ExOfSpe2735}
Let us see an example of constructing a specialization.
Let $\xi = (2, 7) + (3, 5)$,
and let $S$ be the ABS associated to $N_\xi$.
Then $S$ is described as
\\
\vspace{6mm}
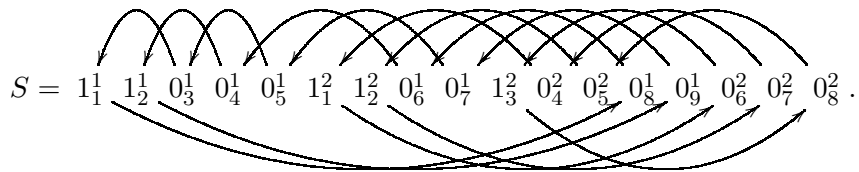
$$S = \xymatrix@=1pt
{
1_{ 1 }^ 1 \ar@/_30pt/[rrrrrrrrrrrr] &
1_{ 2 }^ 1 \ar@/_30pt/[rrrrrrrrrrrr] &
0_{ 3 }^ 1 \ar@/_30pt/[ll] &
0_{ 4 }^ 1 \ar@/_30pt/[ll] &
0_{ 5 }^ 1 \ar@/_30pt/[ll] &
1_{ 1 }^ 2 \ar@/_30pt/[rrrrrrrrr] &
1_{ 2 }^ 2 \ar@/_30pt/[rrrrrrrrr] &
0_{ 6 }^ 1 \ar@/_30pt/[llll] &
0_{ 7 }^ 1 \ar@/_30pt/[llll] &
1_{ 3 }^ 2 \ar@/_30pt/[rrrrrrr] &
0_{ 4 }^ 2 \ar@/_30pt/[lllll] &
0_{ 5 }^ 2 \ar@/_30pt/[lllll] &
0_{ 8 }^ 1 \ar@/_30pt/[lllll] &
0_{ 9 }^ 1 \ar@/_30pt/[lllll] &
0_{ 6 }^ 2 \ar@/_30pt/[lllll] &
0_{ 7 }^ 2 \ar@/_30pt/[lllll] &
0_{ 8 }^ 2 \ar@/_30pt/[lllll]
}.$$
\vspace{4mm}
\\
Let $S'$ denote the specialization obtained by $0^1_4$ and $1^2_2$.
Then $S'$ is described as
\\
\vspace{6mm}
$$S' = \xymatrix@=1pt
{
1_{ 1 }^ 1 \ar@/_30pt/[rrrrrrrrrrrr] &
0_{ 3 }^ 1 \ar@/_30pt/[l] &
1_{ 2 }^ 1 \ar@/_30pt/[rrrrrrrrrrr] &
1_{ 2 }^ 2 \ar@/_30pt/[rrrrrrrrrrr] &
0_{ 5 }^ 1 \ar@/_30pt/[lll] &
1_{ 1 }^ 2 \ar@/_30pt/[rrrrrrrrrr] &
0_{ 4 }^ 1 \ar@/_30pt/[llll] &
0_{ 6 }^ 1 \ar@/_30pt/[llll] &
0_{ 7 }^ 1 \ar@/_30pt/[llll] &
0_{ 4 }^ 2 \ar@/_30pt/[llll] &
1_{ 3 }^ 2 \ar@/_30pt/[rrrrrr] &
0_{ 5 }^ 2 \ar@/_30pt/[lllll] &
0_{ 8 }^ 1 \ar@/_30pt/[lllll] &
0_{ 9 }^ 1 \ar@/_30pt/[lllll] &
0_{ 7 }^ 2 \ar@/_30pt/[lllll] &
0_{ 6 }^ 2 \ar@/_30pt/[lllll] &
0_{ 8 }^ 2 \ar@/_30pt/[lllll]
}.$$
\vspace{4mm}
\\
One can see that these $S$ and $S'$ satisfy $\ell(S') = \ell(S) - 1$,
i.e., this $S'$ is a generic specialization of $S$.
\end{example}

\begin{example} \label{ExOfSpe271235}
Next, let us treat a Newton polygon consisting of three segments.
Let $\xi = (2, 7) + (1, 2) + (3, 5)$.
Then the ABS $S$ corresponding to $N_\xi$ is
\\
\vspace{6mm}
$$S = \xymatrix@=1pt
{
1_{ 1 }^ 1 \ar@/_30pt/[rrrrrrrrrrrrrr] &
1_{ 2 }^ 1 \ar@/_30pt/[rrrrrrrrrrrrrr] &
0_{ 3 }^ 1 \ar@/_30pt/[ll] &
0_{ 4 }^ 1 \ar@/_30pt/[ll] &
0_{ 5 }^ 1 \ar@/_30pt/[ll] &
1_{ 1 }^ 2 \ar@/_30pt/[rrrrrrrrrrr] &
1_{ 1 }^ 3 \ar@/_30pt/[rrrrrrrrrrr] &
1_{ 2 }^ 3 \ar@/_30pt/[rrrrrrrrrrr] &
0_{ 6 }^ 1 \ar@/_30pt/[lllll] &
0_{ 7 }^ 1 \ar@/_30pt/[lllll] &
0_{ 2 }^ 2 \ar@/_30pt/[lllll] &
1_{ 3 }^ 3 \ar@/_30pt/[rrrrrrrr] &
0_{ 4 }^ 3 \ar@/_30pt/[llllll] &
0_{ 5 }^ 3 \ar@/_30pt/[llllll] &
0_{ 8 }^ 1 \ar@/_30pt/[llllll] &
0_{ 9 }^ 1 \ar@/_30pt/[llllll] &
0_{ 3 }^ 2 \ar@/_30pt/[llllll] &
0_{ 6 }^ 3 \ar@/_30pt/[llllll] &
0_{ 7 }^ 3 \ar@/_30pt/[llllll] &
0_{ 8 }^ 3 \ar@/_30pt/[llllll]
}.$$
\vspace{4mm}
\\
For this $S$,
the specialization $S'$ obtained by exchanging $0^1_4$ and $1^3_2$ is 
\\
\vspace{6mm}
$$S' = \xymatrix@=1pt
{
1_{ 1 }^ 1 \ar@/_30pt/[rrrrrrrrrrrrrr] &
0_{ 3 }^ 1 \ar@/_30pt/[l] &
1_{ 2 }^ 1 \ar@/_30pt/[rrrrrrrrrrrrr] &
1_{ 2 }^ 3 \ar@/_30pt/[rrrrrrrrrrrrr] &
0_{ 5 }^ 1 \ar@/_30pt/[lll] &
1_{ 1 }^ 3 \ar@/_30pt/[rrrrrrrrrrrr] &
1_{ 1 }^ 2 \ar@/_30pt/[rrrrrrrrrrrr] &
0_{ 4 }^ 1 \ar@/_30pt/[lllll] &
0_{ 6 }^ 1 \ar@/_30pt/[lllll] &
0_{ 7 }^ 1 \ar@/_30pt/[lllll] &
0_{ 4 }^ 3 \ar@/_30pt/[lllll] &
1_{ 3 }^ 3 \ar@/_30pt/[rrrrrrrr] &
0_{ 2 }^ 2 \ar@/_30pt/[llllll] &
0_{ 5 }^ 3 \ar@/_30pt/[llllll] &
0_{ 8 }^ 1 \ar@/_30pt/[llllll] &
0_{ 9 }^ 1 \ar@/_30pt/[llllll] &
0_{ 7 }^ 3 \ar@/_30pt/[llllll] &
0_{ 6 }^ 3 \ar@/_30pt/[llllll] &
0_{ 3 }^ 2 \ar@/_30pt/[llllll] &
0_{ 8 }^ 3 \ar@/_30pt/[llllll]
}.$$
\vspace{4mm}
\\
We see that this $S'$ is not generic.
\end{example}

For certain Newton polygons $\xi$,
the ABS associated to $N_\xi$ is described as follows:

\begin{lemma} \label{LemOfl212l1}
Let $N_{\xi}$ be the minimal ${\rm DM_1}$ of $\xi = (m_1, n_1) + (m_2, n_2)$
with $\lambda_2 < 1/2 < \lambda_1$.
For the above notation, 
the sequence $S$ associated to $N_\xi$ is obtained by the following:
\begin{eqnarray}
\underbrace{1^1_1\cdots 1^1_{m_1}}_{m_1}
\underbrace{0^1_{m_1+1}\cdots 0^1_{n_1}}_{n_1-m_1}
\underbrace{1^2_1\cdots 1^2_{n_2}}_{n_2}
\underbrace{0^1_{n_1+1}\cdots 0^1_{h_1}}_{m_1}
\underbrace{1^2_{n_2+1}\cdots 1^2_{m_2}}_{m_2-n_2}
\underbrace{0^2_{m_2+1}\cdots 0^2_{h_2}}_{n_2}.
\end{eqnarray}
\end{lemma}

\begin{proof}
See \cite{HarashitaCon}, Proposition~4.20.
\end{proof}

In Construction~\ref{ConstAn} and Construction~\ref{ConstBn},
we introduce a method to construct specializations $S'$ of $S$
combinatorially.
Using this method
we can calculate lengths of specializations,
and classify generic specializations.
For instance, in Proposition~\ref{PropOfrq} and Corollary~\ref{CoroOfrr+1},
using this construction, we give a necessary condition for
a specialization to be generic.

\begin{construction} \label{ConstAn}
Let $S$ be the ABS of a minimal ${\rm DM_1}$.
Let $S^{(0)}$ be the small modification by $0^r_i$ and $1^q_j$.
Set $\delta = \Delta(S^{(0)})$ and $\pi = \Pi(S^{(0)})$.
For non-negative integers $n$, 
we write $\alpha_n$ for $\pi^n(0^r_i)$.
We define a subset $A_0$ of $T(S^{(0)})$ to be
$$A_0 = \{ t \in T(S^{(0)}) \mid t < \alpha_0 \text{ and } \alpha_1 < \pi(t) 
\text{ in } T(S^{(0)}), \text{ with } \delta(t) = 0\}$$
endowed with the order induced from $T(S^{(0)})$.
Let $n$ be a natural number.
We construct an ABS $S^{(n)}$ and a set $A_n$
by the ABS $S^{(n-1)}$ and the set $A_{n-1}$
as follows.
Let $T(S^{(n)}) = T(S^{(n-1)})$ as sets.
We define the order on $T(S^{(n)})$ so that
for $t < t'$ in $S^{(n-1)}$,
we have $t > t'$ if and only if
$\alpha_n < t' \leq \pi(t_{\rm max})$ and $t = \alpha_n$  in $S^{(n-1)}$.
Here $t_{\rm max}$ is the maximum element of $A_{n-1}$.
We define the set $A_n$ by
$$A_n = 
\{ t \in T(S^{(n)}) - T(S_q) \mid t < \alpha_n \text{ and } \alpha_{n+1} < \pi(t)
\text{ in } T(S^{(n)})
\text{ with } \delta(t) = \delta(\alpha_n)\}$$
endowed with the order induced from $S^{(n)}$.
Thus we obtain the ABS $S^{(n)} = (T(S^{(n)}), \delta, \pi)$ and the set $A_n$.
\end{construction}

Proposition~\ref{PropOfAVanish} implies that 
if the specialization obtained by a small modification is generic,
then there exists a non-negative integer $a$ such that
$A_a = \emptyset$.
Now we suppose that there exists such an integer $a$.
Then we can define the following ABS's and sets.

\begin{construction} \label{ConstBn}
For the ABS $S$ corresponding to a minimal ${\rm DM_1}$,
let $S^{(0)}$ be the small modification by $0^r_i$ and $1^q_j$.
We write $\delta$ for $\Delta(S^{(0)})$ and $\pi$ for $\Pi(S^{(0)})$.
Put $\beta_n = \pi^n(1^q_j)$ for non-negative integers $n$.
Assume that there exists the minimum non-negative integer $a$ 
such that $A_a = \emptyset$,
we define a set $B_0$ by
$$B_0 = \{t \in T(S^{(a)}) \mid \beta_0 < t \text{ and } \pi(t) < \beta_1
\text{ in } T(S^{(a)}) \text{ with } \delta(t) = 1\}$$
endowed with the order induced from $T(S^{(a)})$.
For the ABS $S^{(a+n-1)}$ and 
the set $B_{n-1}$, we define 
an ABS $S^{(a+n)}$ as follows.
Let $T(S^{(a+n)}) = T(S^{(a+n-1)})$ as sets.
Let $\Delta(S^{(a+n)}) = \Delta(S^{(a+n-1)})$ 
and $\Pi(S^{(a+n)}) = \Pi(S^{(a+n-1)})$.
The ordering of $T(S^{(a+n)})$ is given so that
for $t < t'$ in $S^{(a+n-1)}$, 
we have $t > t'$ if and only if
$\pi(t_{\rm min}) \leq t < \beta_n$ and $t' = \beta_n$,
where $t_{\rm min}$ is the minimum element of $B_{n-1}$.
We define the set $B_n$ as 
$$B_n = \{t \in T(S^{(a+n)}) \mid \beta_n < t \text{ and } \pi(t) < \beta_{n+1} 
\text{ in } T(S^{(a+n)}) \text{ with } \delta(t) = \delta(\beta_n)\}$$
with the ordering obtained from the order on $S^{(a+n)}$.
Thus we obtain the ABS $S^{(a+n)}$ and the set $B_n$.
\end{construction}

For a small modification by $0^r_i$ and $1^q_j$,
if there exists non-negative integers $a$ and $b$ such that
$A_a = \emptyset$ and $B_b = \emptyset$,
then 
we call the ABS $S^{(a+b)}$ the {\it full modification by $0^r_i$ and $1^q_j$}.
In Proposition~\ref{PropOfBVanish},
we will see that
if the specialization is generic,
then for the above sets $B_n$,
there exists a non-negative integer $b$ such that
$B_b = \emptyset$, i.e., 
the generic specialization is obtained by the ABS $S^{(a+b)}$
for some integers $a$ and $b$.

\begin{example} \label{ExOfFullModi2735}
Let $\xi = (2, 7) + (3, 5)$.
Let $S$ be the ABS of $\xi$.
The small modification $S^{(0)}$ by $0^1_4$ and $1^2_2$
is described as 
$$S^{(0)}: 
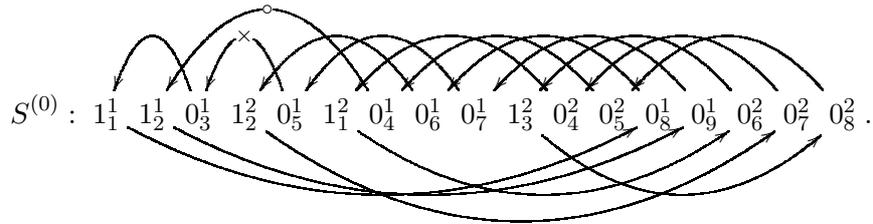
\xymatrix@=1pt
{
1_{ 1 }^ 1 \ar@/_30pt/[rrrrrrrrrrrr] &
1_{ 2 }^ 1 \ar@/_30pt/[rrrrrrrrrrrr] &
0_{ 3 }^ 1 \ar@/_30pt/[ll] &
1_{ 2 }^ 2 \ar@/_40pt/[rrrrrrrrrrrr] &
0_{ 5 }^ 1 \ar@/_30pt/[ll]|\times &
1_{ 1 }^ 2 \ar@/_30pt/[rrrrrrrrr] &
0_{ 4 }^ 1 \ar@/_40pt/[lllll]|\circ &
0_{ 6 }^ 1 \ar@/_30pt/[llll] &
0_{ 7 }^ 1 \ar@/_30pt/[llll] &
1_{ 3 }^ 2 \ar@/_30pt/[rrrrrrr] &
0_{ 4 }^ 2 \ar@/_30pt/[lllll] &
0_{ 5 }^ 2 \ar@/_30pt/[lllll] &
0_{ 8 }^ 1 \ar@/_30pt/[lllll] &
0_{ 9 }^ 1 \ar@/_30pt/[lllll] &
0_{ 6 }^ 2 \ar@/_30pt/[lllll] &
0_{ 7 }^ 2 \ar@/_30pt/[lllll] &
0_{ 8 }^ 2 \ar@/_30pt/[lllll]
}.$$
\vspace{6mm}
\\
For this $S^{(0)}$, we have sets
$A_0 = \{0^1_5\}$ and $A_1 = \emptyset$.
The ABS $S^{(1)}$ is obtained by
\\
\\
\vspace{6mm}
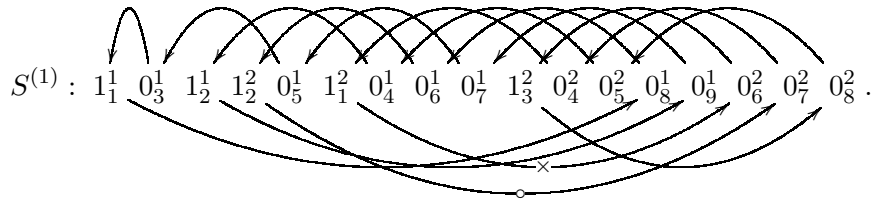
$$S^{(1)}: \xymatrix@=1pt
{
1_{ 1 }^ 1 \ar@/_30pt/[rrrrrrrrrrrr] &
0_{ 3 }^ 1 \ar@/_30pt/[l] &
1_{ 2 }^ 1 \ar@/_30pt/[rrrrrrrrrrr] &
1_{ 2 }^ 2 \ar@/_40pt/[rrrrrrrrrrrr]|\circ &
0_{ 5 }^ 1 \ar@/_30pt/[lll] &
1_{ 1 }^ 2 \ar@/_30pt/[rrrrrrrrr]|\times &
0_{ 4 }^ 1 \ar@/_30pt/[llll] &
0_{ 6 }^ 1 \ar@/_30pt/[llll] &
0_{ 7 }^ 1 \ar@/_30pt/[llll] &
1_{ 3 }^ 2 \ar@/_30pt/[rrrrrrr] &
0_{ 4 }^ 2 \ar@/_30pt/[lllll] &
0_{ 5 }^ 2 \ar@/_30pt/[lllll] &
0_{ 8 }^ 1 \ar@/_30pt/[lllll] &
0_{ 9 }^ 1 \ar@/_30pt/[lllll] &
0_{ 6 }^ 2 \ar@/_30pt/[lllll] &
0_{ 7 }^ 2 \ar@/_30pt/[lllll] &
0_{ 8 }^ 2 \ar@/_30pt/[lllll]
}.$$
\\
By the above
$B_0 = \{1^2_1\}$. 
We have $B_1 = \{0^2_7\}$ with the ABS
$$S^{(2)}: 
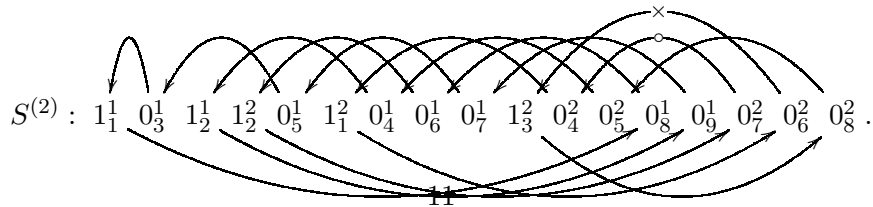
\xymatrix@=1pt
{
1_{ 1 }^ 1 \ar@/_30pt/[rrrrrrrrrrrr] &
0_{ 3 }^ 1 \ar@/_30pt/[l] &
1_{ 2 }^ 1 \ar@/_30pt/[rrrrrrrrrrr] &
1_{ 2 }^ 2 \ar@/_30pt/[rrrrrrrrrrr] &
0_{ 5 }^ 1 \ar@/_30pt/[lll] &
1_{ 1 }^ 2 \ar@/_30pt/[rrrrrrrrrr] &
0_{ 4 }^ 1 \ar@/_30pt/[llll] &
0_{ 6 }^ 1 \ar@/_30pt/[llll] &
0_{ 7 }^ 1 \ar@/_30pt/[llll] &
1_{ 3 }^ 2 \ar@/_30pt/[rrrrrrr] &
0_{ 4 }^ 2 \ar@/_30pt/[lllll] &
0_{ 5 }^ 2 \ar@/_30pt/[lllll] &
0_{ 8 }^ 1 \ar@/_30pt/[lllll] &
0_{ 9 }^ 1 \ar@/_30pt/[lllll] &
0_{ 7 }^ 2 \ar@/_30pt/[llll]|\circ &
0_{ 6 }^ 2 \ar@/_40pt/[llllll]|\times &
0_{ 8 }^ 2 \ar@/_30pt/[lllll]
}.$$
\vspace{6mm}
\\
Clearly $B_2 = \emptyset$.
Hence we see $a = 1$ and $b = 1$.
One can check that the full modification $S^{(3)}$ is equal to $S'$
of Example~\ref{ExOfSpe2735}.
\end{example}

\begin{example} \label{ExOfFullModi271235}
For the ABS $S$ of $\xi = (2, 7) + (1, 2) + (3, 5)$,
Consider the small modification by $0^1_4$ and $1^3_2$.
Then the ABS $S^{(0)}$ is 
$$S^{(0)}: 
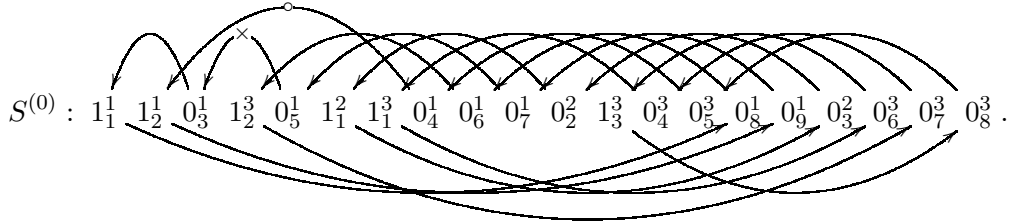
\xymatrix@=1pt
{
1_{ 1 }^ 1 \ar@/_30pt/[rrrrrrrrrrrrrr] &
1_{ 2 }^ 1 \ar@/_30pt/[rrrrrrrrrrrrrr] &
0_{ 3 }^ 1 \ar@/_30pt/[ll] &
1_{ 2 }^ 3 \ar@/_40pt/[rrrrrrrrrrrrrrr] &
0_{ 5 }^ 1 \ar@/_30pt/[ll]|\times &
1_{ 1 }^ 2 \ar@/_30pt/[rrrrrrrrrrr] &
1_{ 1 }^ 3 \ar@/_30pt/[rrrrrrrrrrr] &
0_{ 4 }^ 1 \ar@/_40pt/[llllll]|\circ &
0_{ 6 }^ 1 \ar@/_30pt/[lllll] &
0_{ 7 }^ 1 \ar@/_30pt/[lllll] &
0_{ 2 }^ 2 \ar@/_30pt/[lllll] &
1_{ 3 }^ 3 \ar@/_30pt/[rrrrrrrr] &
0_{ 4 }^ 3 \ar@/_30pt/[llllll] &
0_{ 5 }^ 3 \ar@/_30pt/[llllll] &
0_{ 8 }^ 1 \ar@/_30pt/[llllll] &
0_{ 9 }^ 1 \ar@/_30pt/[llllll] &
0_{ 3 }^ 2 \ar@/_30pt/[llllll] &
0_{ 6 }^ 3 \ar@/_30pt/[llllll] &
0_{ 7 }^ 3 \ar@/_30pt/[llllll] &
0_{ 8 }^ 3 \ar@/_30pt/[llllll]
}.$$
\vspace{6mm}
\\
We have
$A_0 = \{0^1_5\}$ and $A_1 = \emptyset$.
Thus we see $a = 1$.
We have the set $B_0 = \{1^2_1, 1^3_1\}$ and $B_1 = \{0^2_3, 0^3_6\}$
with the ABS $S^{(2)}$
$$S^{(2)}: \xymatrix@=1pt
{
1_{ 1 }^ 1 \ar@/_30pt/[rrrrrrrrrrrrrr] &
0_{ 3 }^ 1 \ar@/_30pt/[l] &
1_{ 2 }^ 1 \ar@/_30pt/[rrrrrrrrrrrrr] &
1_{ 2 }^ 3 \ar@/_30pt/[rrrrrrrrrrrrr] &
0_{ 5 }^ 1 \ar@/_30pt/[lll]&
1_{ 1 }^ 2 \ar@/_30pt/[rrrrrrrrrrrr] &
1_{ 1 }^ 3 \ar@/_30pt/[rrrrrrrrrrrr] &
0_{ 4 }^ 1 \ar@/_30pt/[lllll]&
0_{ 6 }^ 1 \ar@/_30pt/[lllll] &
0_{ 7 }^ 1 \ar@/_30pt/[lllll] &
0_{ 2 }^ 2 \ar@/_30pt/[lllll] &
1_{ 3 }^ 3 \ar@/_30pt/[rrrrrrrr] &
0_{ 4 }^ 3 \ar@/_30pt/[llllll] &
0_{ 5 }^ 3 \ar@/_30pt/[llllll] &
0_{ 8 }^ 1 \ar@/_30pt/[llllll] &
0_{ 9 }^ 1 \ar@/_30pt/[llllll] &
0_{ 7 }^ 3 \ar@/_30pt/[llll]|\circ &
0_{ 3 }^ 2 \ar@/_40pt/[lllllll]|\times &
0_{ 6 }^ 3 \ar@/_40pt/[lllllll]|\times &
0_{ 8 }^ 3 \ar@/_30pt/[llllll]
}.$$
\vspace{6mm}
\\
By the ABS $S^{(2)}$, we obtain the set $B_2 = \{0^2_2\}$
and the ABS
$$S^{(3)}: 
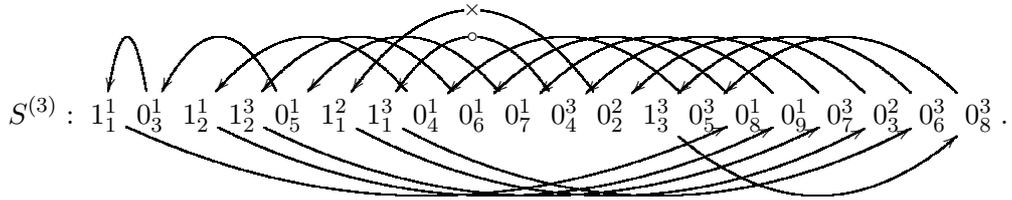
\xymatrix@=1pt
{
1_{ 1 }^ 1 \ar@/_30pt/[rrrrrrrrrrrrrr] &
0_{ 3 }^ 1 \ar@/_30pt/[l] &
1_{ 2 }^ 1 \ar@/_30pt/[rrrrrrrrrrrrr] &
1_{ 2 }^ 3 \ar@/_30pt/[rrrrrrrrrrrrr] &
0_{ 5 }^ 1 \ar@/_30pt/[lll]&
1_{ 1 }^ 2 \ar@/_30pt/[rrrrrrrrrrrr] &
1_{ 1 }^ 3 \ar@/_30pt/[rrrrrrrrrrrr] &
0_{ 4 }^ 1 \ar@/_30pt/[lllll]&
0_{ 6 }^ 1 \ar@/_30pt/[lllll] &
0_{ 7 }^ 1 \ar@/_30pt/[lllll] &
0_{ 4 }^ 3 \ar@/_30pt/[llll]|\circ &
0_{ 2 }^ 2 \ar@/_40pt/[llllll]|\times &
1_{ 3 }^ 3 \ar@/_30pt/[rrrrrrr] &
0_{ 5 }^ 3 \ar@/_30pt/[llllll] &
0_{ 8 }^ 1 \ar@/_30pt/[llllll] &
0_{ 9 }^ 1 \ar@/_30pt/[llllll] &
0_{ 7 }^ 3 \ar@/_30pt/[llllll] &
0_{ 3 }^ 2 \ar@/_30pt/[llllll] &
0_{ 6 }^ 3 \ar@/_30pt/[llllll] &
0_{ 8 }^ 3 \ar@/_30pt/[llllll]
}.$$
\vspace{6mm}
\\
Similarly, we obtain $B_3 = \{1^2_1\}$, $B_4 = \{0^2_3\}$ and $B_5 = \emptyset$
with the ABS's $S^{(4)}$, $S^{(5)}$ and $S^{(6)}$.
Hence we have $b = 5$,
and the full modification $S^{(6)}$ is equal to $S'$ of Example~\ref{ExOfSpe271235}.
\end{example}

\section{The case of arbitrary Newton polygons} \label{ArbitNP}

In this section,
we will prove Theorem~\ref{ThmOfOneToOne}.
The steps of the proof is as follows.
Firstly, in Section~\ref{SomePropOfNP},
we introduce some operations of Newton polygons
which reduce the problem to the case that
Newton polygons satisfy a special property.
Secondly, in Section~\ref{ConstOfFullModifi},
we introduce a combinatorial way to construct 
generic specializations.
Finally, in Section~\ref{Proof}, we prove Theorem~\ref{ThmOfOneToOne}.

\subsection{Euclidean algorithm for Newton polygons} \label{SomePropOfNP}

Here, we introduce some properties of Newton polygons
and corresponding ${\rm DM_1}$'s.
First, we define two operations of Newton polygons,
which are used in Section~\ref{SomePropOfNP}.
See Section~\ref{Definitions} \eqref{EqOfNP}
for the notation of Newton polygons.
Let $\xi = \sum_{i=1}^z (m_i, n_i)$ be a Newton polygon.
We define the Newton polygon $\xi^{\rm D}$ by
$$\xi^{\rm D} = \sum_{i = 1}^z (n_{z - i + 1}, m_{z - i + 1}).$$
We call this $\xi^{\rm D}$ the {\it dual of $\xi$}.
Moreover, for a Newton polygon $\xi$ satisfying $m_i \leq n_i$ for all $i$, 
we define the Newton polygon $\xi^{\rm C}$ by
$$\xi^{\rm C} = \sum_{i = 1}^z (m_i, n_i - m_i),$$
and we call this $\xi^{\rm C}$ the {\it curtailment of $\xi$}.

We denote by ${\rm NP}$ the set of Newton polygons.
Let ${\rm NP^{sep}}$ be the subset of ${\rm NP}$ consisting of 
Newton polygons $(m_1, n_1) + (m_2, n_2) + \dots + (m_z, n_z)$
with $n_z/(m_z + n_z) < 1/2 < n_1/(m_1 + n_1)$.
Using the above, 
we construct a map $\Phi: {\rm NP} \to {\rm NP^{sep}}$
which allow us to suppose that the Newton polygon $\xi$ consists of 
two segments and $\xi \in {\rm NP^{sep}}$ for the set $B(\xi)$.
Let $\xi = \sum_{i=1}^z (m_i, n_i)$ be a Newton polygon.
We may assume that $m_i \leq n_i$ for all $i$ by the duality.
First, we treat the case $z = 2$.
Let $q_i$ and $r_i$ be the non-negative integers such that
$n_i = q_i m_i + r_i$ with $r_i < m_i$,
i.e., we denote by $q_i$ and $r_i$ the quotient and the reminder
obtained by dividing $n_i$ by $m_i$.
As $\lambda_2 < \lambda_1$, we have $q_2 \leq q_1$.
In fact, $\lambda_2 < \lambda_1$ implies that 
$m_1 m_2(q_2 - q_1) < m_2 r_1 - m_1 r_2$.
Suppose $q_1 < q_2$.
Then $m_1 m_2 < m_2r_1 - m_1r_2$ holds.
It induces that $m_2(r_1 - m_1) > m_1r_2 > 0$.
This is a contradiction.

Let us construct the Newton polygon $\Phi(\xi)$.
If $q_2 < q_1$, then constructing curtailments, 
we obtain the desired Newton polygon 
$\Phi(\xi)$ by $\Phi(\xi) = (m_1, n_1 - q_2 m_1) + (m_2, r_2)$.
Assume that $q_1 = q_2$.
Constructing curtailments, we obtain the Newton polygon $(m_1, r_1) + (m_2, r_2)$.
Take the dual $(r_2, m_2) + (r_1, m_1)$, and
we can repeat constructing curtailments.
We may assume $r_1 = 1$ or $r_2 = 1$ since $\gcd(m_i, n_i) = 1$.
We divide the proof into the three cases:
\begin{itemize}
\item[(i)] $\xi = (1, n_1) + (1, n_2)$;
\item[(ii)] $\xi = (m_1, n_1) + (1, n_2)$;
\item[(iii)] $\xi = (1, n_1) + (m_2, n_2)$.
\end{itemize}

In the case (i), if $n_1 > n_2 + 1$, then
we obtain $\Phi(\xi) = (1, n_2 - n_1) + (1, 0)$.
If $n_2 = n_1 + 1$, then by the Newton polygon $(1, 1) + (1, 0)$
obtained by curtailments, 
we take the dual, and we get $\Phi(\xi) = (0, 1) + (1, 0)$.

In the case (ii), we have $q_2 = n_2$.
if $q_1 = q_2$, then by the Newton polygon $(m_1, r_1) + (1, 0)$
obtained by curtailments, constructing the dual and 
curtailments we get $\Phi(\xi) = (0, 1) + (r_1, r)$, where 
$r$ is the reminder obtained by dividing $m_1$ by $r_1$.

In the case (iii), if $q_2 = q_1 = n_1$, then we have $n_2 = n_1 m_2 + r_2$.
On the other hand, it follows that $n_1 m_2 - n_2 > 0$
by the condition of the Newton polygon $\xi$.
It implies $r_2 < 0$, and we have a contradiction.
If $q_1 > q_2 + 1$, then we get $\Phi(\xi) = (1, n_1 - q_2) + (m_2, r_2)$.
If $q_1 = q_2 + 1$, then by curtailments and the dual we have 
the Newton polygon $(r_2, m_2) + (1, 1)$.
By the case (i) and (ii), we can get the desired Newton polygon $\Phi(\xi)$
of this Newton polygon.

Finally, we treat the case $z > 2$.
We have the Newton polygon $\Phi(\eta) = (c_1, d_1) + (c_z, d_z)$ of 
$\eta = (m_1, n_1) + (m_z, n_z)$.
Apply the same operation
constructing $\Phi(\eta)$ from $\eta$ to $\xi$,
and we get the desired $\Phi(\xi) = \sum_{i=1}^z (c_i, d_i)$.



\begin{remark}
By the above construction, the Newton polygon $\Phi(\xi)$
is described as $\Phi(\xi) = \xi^{{\rm Q}_1 {\rm Q}_2 \cdots {\rm Q}_m}$,
where ${\rm Q}_i$ is either the operation ${\rm C}$ or the operation ${\rm D}$ 
for every $i$.
Thus by the duality and Theorem~\ref{ThmOfSpeNP},
for all Newton polygons $\xi$ consisting of two segments,
we obtain a bijection from $B(\xi)$ to $B(\Phi(\xi))$.
\end{remark}

Thanks to this map $\Phi$,
we can partially describe the construction of the ABS
corresponding to a Newton polygon.
We state it as follows:

\begin{proposition} \label{PropOfTwoTerms}
Let $S$ be the ABS of a minimal ${\rm DM_1}$ $N_\xi$
with $\xi = \sum_{i = 1}^z (m_i, n_i)$.
For natural numbers $r$ and $q$ with $r < q \leq z$, we have 
\begin{itemize}
\item[(i)] $1^r_1 < 1^q_1$,
\item[(ii)] $0^r_{m_r+n_r} < 0^q_{m_q+n_q}$,
\item[(iii)] $0^r_{m_r+1} < 0^q_{m_q+1}$
\end{itemize}
in the set $T(S)$.
\end{proposition}

\begin{proof}
It suffices to show (i). In fact,
considering the dual $\xi^{\rm D}$ of $\xi$,
if (i) is true, then (ii) holds.
Moreover, if (i) holds, then we get (iii) since
$0^r_{m_r+1}$ and $0^q_{m_q+1}$ are the inverse images of
$1^r_1$ and $1^q_1$ respectively by $\Pi(S)$.

For a Newton polygon $\xi$,
let ${\rm P}(\xi)$ denote the assertion:
{\it The ABS associated to the minimal ${ DM_1}$ $N_\xi$ satisfies (i).}
To show the lemma, it suffices to deal with the case $z = 2$.
By Proposition~\ref{LemOfl212l1}, if $\xi$ satisfies that 
$\lambda_2 < 1/2 < \lambda_1$,
then ${\rm P}(\xi)$ holds.
To show that ${\rm P}(\xi)$ is true for all Newton polygons $\xi$,
we claim
\begin{itemize}
\item[(A)] If ${\rm P}(\xi^{\rm D})$ holds, then ${\rm P}(\xi)$ also holds;
\item[(B)] If ${\rm P}(\xi^{\rm C})$ holds, then ${\rm P}(\xi)$ also holds.
\end{itemize}
The claim (A) is obvious by the duality.
Moreover, By the construction of $N_\xi$ and $N_{\xi^{\rm C}}$,
we see that (B) holds.
The assertion of the lemma follows from 
(A), (B) and Section~\ref{SomePropOfNP}.
\end{proof}

Let $S$ (resp. $R$) be the ABS of $\xi$ (resp. $\xi^{\rm C}$).
Next, we describe a relation between $S$ and $R$.
In the following lemma, we show that
the set $T(R)$ can be regarded as a subset of $T(S)$
as ordered sets.
This relation is used for the proof of Lemma~\ref{LemOf01}
and Theorem~\ref{ThmOfSpeNP}.

\begin{lemma} \label{LemOfXiAndXiC}
Let $\xi = (m_1, n_1) + (m_2, n_2)$ be a Newton polygon
consisting of two segments
with $m_i \leq n_i$ for $i = 1, 2$.
Let $S$ and $R$ be the ABS of $\xi$ and $\xi^{\rm C}$ respectively.
Then $T(R)$ is contained in $T(S)$ as an ordered set.
We have
\begin{eqnarray} \label{EqOfTR1AndTS1}
\{t \in T(R) \mid \Delta(R)(t) = 1\} = \{t \in T(S) \mid \Delta(S)(t) = 1\}
\end{eqnarray}
and 
\begin{eqnarray}
T(S) - T(R) = \{\Pi(S)(t) \mid t \in T(S) \text{ with } \Delta(S)(t) = 1\}.
\end{eqnarray}
Let $t$ be an element of $T(R) \subset T(S)$.
Then 
\begin{eqnarray}
\Pi(R)(t) =
\begin{cases}
\Pi(S)(t) & \text{if }\Delta(R)(t) = 0,\\
\Pi(S)^2(t) & \text{otherwise}.
\end{cases}
\end{eqnarray}
holds.
\end{lemma}

\begin{proof}
The assertion follows from Remark~\ref{RemOfSimpleABS}
and Definition~\ref{DefOfDirectSum}.
\end{proof}


\subsection{Combinatorial construction of generic specializations}
\label{ConstOfFullModifi}

In this section, using Construction~\ref{ConstAn} and Construction~\ref{ConstBn},
we determine the {\it full modification}, which is a specialization 
constructed combinatorially,
for a given small modification.

We use the notation of Notation~\ref{NotOfNxi}.
Furthermore, we fix the following notation.
Let $S$ be the ABS associated to $N_\xi$.
Let $S^{(0)}$ be the small modification by $0^r_i$ and $1^q_j$.
Then we obtain arrowed binary sequences $S^{(1)}, S^{(2)}, \dots$ and 
sets $A_0, A_1, \dots$ by Construction~\ref{ConstAn}.
Put $\delta = \Delta(S^{(0)})$ and $\pi = \Pi(S^{(0)})$.
We set $\alpha_n = \pi^n(0^r_i)$
and $\beta_n = \pi^n(1^q_j)$
for non-negative integers $n$.

\begin{proposition} \label{PropOfAn}
Let $n$ be a natural number.
The set $A_n$ is given by
$$A_n 
= \{ \pi(t) \mid t \in A_{n-1},\ \pi(t) \not \in T(S_q) 
\text{ and } \delta(\pi(t)) = \delta(\alpha_n) \}.$$
\end{proposition}

\begin{proof}
First, take an element $t$ of $A_{n-1}$.
Let us show that if $\pi(t)$ satisfies 
$t \not \in T(S_q)$ and $\delta(t) = \delta(\alpha_n)$,
then $\pi(t)$ belongs to $A_n$.
We have $\alpha _{n+1} < \pi(\pi(t))$
in $T(S^{(n-1)})$ and $T(S^{(n)})$.
Furthermore, by construction, 
$\pi(t) < \alpha_n$ holds in $S^{(n)}$.
Hence $\pi(t)$ belongs to $A_n$.
Conversely, let $\pi(t)$ be an element of $A_n$.
It suffices to show that $t$ belongs to $A_{n-1}$.
In $T(S^{(n-1)})$ and $T(S^{(n)})$, we have $t < \alpha_{n-1}$.
Moreover, by construction, in $T(S^{(n-1)})$, 
we have $\alpha_n < \pi(t)$.
Hence we see that $t$ belongs to $A_{n-1}$.
\end{proof}

\begin{proposition} \label{PropOfalphamn}
$A_{n}$ does not contain elements $\alpha_m$ for $m \leq n$.
\end{proposition}

\begin{proof}
Note that for all non-negative integers $n$,
sets $A_n$ do not contain the inverse image of 
$\alpha_0$, which is an the element of $T(S_q)$.
Here $S_q$ is the ABS corresponding to the ${\rm DM_1}$ $N_{m_q, n_q}$.
Let us show the assertion by induction on $n$.
The case $n = 0$ is obvious.
For a natural number $n$, 
suppose that $A_n$ contains $\alpha_m$
for a non-negative integer $m$ with $m \leq n$.
By Proposition~\ref{PropOfAn},
then $A_{n-1}$ contains $\alpha_{m-1}$.
This contradicts with the hypothesis of induction.
\end{proof}  

\begin{proposition} \label{PropOfAVanish}
If there exists no non-negative integer $a$
such that $A_a = \emptyset$,
then the specialization $S'$ is not generic.
\end{proposition}

\begin{proof}
Let $a'$ be the minimum number satisfying $\alpha_{a'} = \beta_0$.
We define the subset $B_{-1}$ of $T(S^{(0)})$ by 
$$B_{-1} = \{t \in T(S^{(0)}) \mid \beta_0 < t \text{ and } \pi(t) < \beta_1,
\text{ with } \delta(t) = 1\}.$$
Now we claim that if an element $t$ of $T(S^{a'-1})$
belongs to $B_{-1}$,
then $t' := \pi^{-1}(t)$ belongs to $A_{a'-1}$.
In fact, if $t'$ does not belong to $A_{a'-1}$,
then all elements $t''$ of $A_{a'-1}$ satisfy $t'' < t'$
in the set $T(S^{(a'-1)})$.
We have then $A_{a'} = \emptyset$,
and this is a contradiction.
Thus we see that the set 
$$\{t \in T(S^{(a')}) \mid \beta_0 < t \text{ and } \pi(t) < \beta_1, 
\text{ with } \delta(t) = 1\}$$
is empty.
If there exists no non-negative integer $a$ such that $A_a$ is empty,
then we have a non-negative integer $m$ satisfying that $|A_m| = |A_{m+1}| = \cdots$.
The elements of $T(S^{(m)})$ are 
ordered by binary expansions determined by $\pi$.
Thus we obtain $S'$ by $S' = (T(S^{(m)}), \delta, \pi)$.
We see that $\ell(S^{(m)}) = \ell(S')$.

Here, let us compare lengths of $S$ and $S'$.
Put $$B_0' = \{t \in T(S^{(0)}) \mid \beta_0 < t \text{ and } \pi(t) < \beta_1
\text{ in } T(S^{(0)}) \text{ with } \delta(t) = 1\}.$$
We have $\ell(S) - \ell(S^{(0)}) = |A_0| + |B_0'| + 1$.
Since $\ell(S^{(m)}) - \ell(S^{(0)}) \leq |A_0| - |A_m|$,
we see $\ell(S') < \ell(S) - 1$.
\end{proof}

By Proposition~\ref{PropOfAVanish},
we may assume that there exists a non-negative integer $a$ such that
$A_a$ is an empty set
to classify generic specializations of arrowed binary sequences.

In Proposition~\ref{PropOfBVanish},
we will show that to classify generic specializations,
it suffices to consider the case
there exists a non-negative integer $b$ such that
$B_b = \emptyset$.
Let us see a property of sets $B_n$,
which is used for the proof of Proposition~\ref{PropOfBVanish}.

\begin{proposition} \label{PropOfBn}
Let $n$ be a natural number.
The set $B_n$ is obtained by
$$B_n = \{\pi(t) \mid t \in B_{n-1} \text{ and } \delta(\pi(t)) = \delta(\beta_n)\}.$$
\end{proposition}

\begin{proof}
A proof is given by the same way as Proposition~\ref{PropOfAn}.
\end{proof}

\begin{proposition} \label{PropOfBVanish}
If there exists no non-negative integer $b$ such that
$B_b = \emptyset$, then the specialization is not generic.
\end{proposition}

\begin{proof}
In this hypothesis, there exists a non-negative integer $m$ such that
$|B_m| = |B_{m+1}| = \cdots.$
Then the elements of $T(S^{(a+m)})$ are ordered by
these binary expansions.
Thus we obtain the specialization $S' = (T(S^{(a+m)}), \delta, \pi)$. 

Let us compare lengths of $S$ and $S'$.
We have $\ell(S) - \ell(S^{(a)}) \geq |B_0| + 1$.
Since $\ell(S^{(a+m)}) - \ell(S^{(a)}) = |B_0| - |B_m|$
with $|B_m| > 0$, we have
$$\ell(S') \leq \ell(S) - 1 - |B_m|.$$
Thus we see that $\ell(S') < \ell(S) - 1$.
\end{proof}

By Proposition~\ref{PropOfAVanish} and Proposition~\ref{PropOfBVanish},
to classify generic specializations,
we may suppose that there exist non-negative integers $a$ and $b$
such that $A_a = \emptyset$ and $B_b = \emptyset$
for a small modification.
For the ABS $S^{(a+b)}$,
if elements $t$ and $t'$ of $T(S^{(a+b)})$ satisfy that
$t < t'$ and $\delta(t) = \delta(t')$,
then $\pi(t) < \pi(t')$ holds.
Thus we see that by Construction~\ref{ConstAn} and Construction~\ref{ConstBn},
for a small modification, we get a specialization $S' = S^{(a+b)}$ of $S$.
We call this ABS $S^{(a+b)}$ the full modification by $0^r_i$ and $1^q_j$.

\subsection{Proof of Theorem~\ref{ThmOfOneToOne}} \label{Proof}

The main purpose of this section is to prove Theorem~\ref{ThmOfOneToOne}.
The notation is as above.
Let $S^{(0)}$ denote the small modification by $0^r_i$ and $1^q_j$.
Lemma~\ref{LemOf01} and Proposition~\ref{PropOfrq} imply that
to classify generic specializations, we may suppose that $q = r + 1$;
see Corollary~\ref{CoroOfrr+1}.

\begin{lemma} \label{LemOf01}
Let $S$ be the ABS associated to $N_\xi$ 
with $\xi = \sum_{i=1}^z (m_i, n_i)$ a Newton polygon.
Let $0^r$ and $1^q$ be elements of $T(S)$ satisfying that $r + 1 < q$
and $0^r < 1^q$ in $T(S)$.
Then there exists an element $t^x$ of $T(S)$ such that
$r < x < q$ and $0^r < t^x < 1^q$.
\end{lemma}

\begin{proof}
For a Newton polygon $\xi$, 
we write ${\rm Q}(\xi)$ for the assertion:
{\it For elements $0^r$ and $1^q$ of the ABS associated to $N_\xi$
satisfying that $r + 1 < q$ and $0^r < 1^q$,
there exists an element $t^x$ such that $r < x < q$
and $0^r < t^x < 1^q$.}
It suffices to treat the case $z = 3$,
$r = 1$ and $q = 3$.
If $\lambda_1 = \lambda_2$ (resp. $\lambda_2 = \lambda_3$) holds,
then we immediately have the desired element $t^x$
since for elements $0^1_i < 1^3_j$, 
the element $0^2_i$ (resp. $1^2_j$) satisfies $0^1_i < 0^2_i < 1^3_j$
(resp. $0^1_i < 1^2_j < 1^3_j$).
From now on, we assume that the slopes are different from each other.

Now we treat Newton polygons satisfying one of the following:
\begin{itemize}
\item[(i)] $\lambda_3 < 1/2 \leq \lambda_2 < \lambda_1$,
\item[(ii)] $\lambda_3 < \lambda_2 \leq 1/2 < \lambda_1$.
\end{itemize}
By the duality, if ${\rm Q}(\xi)$ is true for all $\xi$ satisfying (i),
then ${\rm Q}(\xi)$ holds for all $\xi$ satisfying (ii).
Suppose that $\xi$ satisfies (i).
Put $h_x = m_x + n_x$ for all $x$.
By Lemma~\ref{LemOfl212l1}, in the ABS of $N_{(m_1, n_1) + (m_3, n_3)}$,
there exists no element $t$ satisfying that 
$0^1_{n_1} < t < 1^3_1$ or $0^1_{h_1} < t < 1^3_{n_3+1}$.
Hence it is enough to show that there exist elements 
$t^2_x$ and $t^2_y$ such that
$0^1_{n_1} < t^2_x < 1^3_1$ and $0^1_{h_1} < t^2_y < 1^3_{n_3+1}$.
If $\lambda_2 > 1/2$, then these elements are obtained by 
$t^2_x = 0^2_{n_2}$ and $t^2_y = 0^2_{h_2}$.
In fact, by Proposition~\ref{PropOfTwoTerms} (ii),
we have $0^1_{n_1} < 0^2_{n_2}$ and $0^1_{h_1} < 0^2_{h_2}$.
Moreover, by the construction of the ABS 
corresponding to the ${\rm DM_1}$ $N_{(m_2, n_2) + (m_3, n_3)}$
we have $0^2_{n_2} < 1^3_1$ and $0^2_{h_2} < 1^3_{n_3+1}$.
If $\lambda_2 = 1/2$, then the desired elements 
$t^2_x$ and $t^2_y$ are obtained by $1^2_1$ and $0^2_2$.

We claim that
\begin{itemize}
\item[(A)] If ${\rm Q}(\xi^{\rm D})$ holds, then ${\rm Q}(\xi)$ also holds;
\item[(B)] If ${\rm Q}(\xi^{\rm C})$ holds, then ${\rm Q}(\xi)$ also holds.
\end{itemize}
If the claim (A) and (B) are true, then by Section~\ref{SomePropOfNP},
the proposition is reduced to the case (i) or (ii),
and we complete the proof.
The claim (A) is obvious.
Let us show (B).
Let $S$ (resp. $R$) denote the ABS associated to $\xi$ (resp. $\xi^{\rm C}$).
We can regard $T(R)$ as a subset of $T(S)$.
Let $U$ (resp. $V$) be the subset of $T(R) \times T(R)$ (resp. $T(S) \times T(S)$)
consisting of the pair $(0^1, 1^3)$ of elements of $T(R)$ (resp. $T(S)$) 
satisfying $0^1 < 1^3$.
By Lemma~\ref{LemOfXiAndXiC}
we have $U = V$,
whence (B) holds.
\end{proof} 

In Proposition~\ref{PropOfrq} and Corollary~\ref{CoroOfrr+1},
we give necessary conditions for specializations to be generic.

\begin{proposition} \label{PropOfrq}
For the small modification by $0^r$ and $1^q$,
if either of the following assertions
\begin{itemize}
\item[(i)] the set $A_0$ contains an element $0^x$ with $r < x$, or
\item[(ii)] the set $B_0$ contains an element $1^x$ with $x < q$,
\end{itemize}
holds, then the specialization is not generic.
\end{proposition}

\begin{proof}
Let $S^{(0)}$ be the small modification by $0^r_i$ and $1^q_j$.
Put $\pi = \Pi(S^{(0)})$.
Set $\alpha_n = \pi^n(0^r_i)$ and $\beta_n = \pi^n(1^q_j)$.
By Proposition~\ref{PropOfAVanish} and Proposition~\ref{PropOfBVanish},
we may assume that there exists the full modification $S^{(a+b)}$
by $0^r_i$ and $1^q_j$.
Put $$B_0' = \{t \in T(S^{(0)}) \mid \beta_0 < t \text{ and }
\pi(t) < \beta_1 \text{ in } S^{(0)} \text{ with } \delta(t) = 1\}.$$
For this set, $\ell(S) - \ell(S^{(0)}) = |A_0| + |B_0'| + 1$.
We have $\ell(S^{(n+1)}) - \ell(S^{(n)}) \leq d(n)$, where
$$
d(n) = 
\begin{cases}
|A_n| - |A_{n+1}| & \text{if } n < a, \\
|B_n| - |B_{n+1}| & \text{if } n \geq a.
\end{cases}
$$
Clearly $\ell(S') - \ell(S^{(0)}) \leq |A_0| + |B_0|$ holds.
First, we show that $\ell(S') - \ell(S^{(0)}) \leq |A_0| + |B_0'|$.
Let $I$ be the subset of $B_0$ consisting of elements 
which are of the form $\alpha_m$.
We have then $|B_0| \leq |B_0'| + |I|$.
Let $m$ be a non-negative integer such that
$A_m$ contains the inverse image of $\beta_0$.
We have then $\ell(S^{(m+1)}) - \ell(S^{(m)}) = d(m) - 1$.
Moreover, if $\delta(\alpha_{m+1}) = 1$, 
then $\alpha_{m+1}$ belongs to $I$.
Hence we see $\ell(S^{(a)}) - \ell(S^{(0)}) \leq |A_0| - |I|$.
Moreover, we have $\ell(S') - \ell(S^{(a)}) = |B_0|$.
Thus we get the desired inequality.

Let us see that in the case (i) the specialization is not generic.
Let $m$ be the minimum number such that 
the set $A_m$ contains no element $t^x$ with $r < x$.
Fix an element $t^x$ of $A_{m-1}$.
Put $t = \pi(t^x)$.
Now we claim that $\delta(t) = 1$ and $\delta(\alpha_m) = 0$.
If $\delta(t) = 0$ and $\delta(\alpha_m) = 1$ is true,
then there exists an element $1^x$ satisfying 
$\alpha_m < 1^x < t$ in $T(S)$.
In fact, if $1^x_n < \alpha_m$ holds in $T(S)$ for all $n$,
then we have $1^x_{m_x} < 1^r_{m_r}$ with $r < x$.
By Proposition~\ref{PropOfTwoTerms} this is a contradiction.
Thus we see that the set $A_m$ contains the element $1^x$.
This contradicts with the minimality of $m$.
Hence we have $\delta(t) = 1$ and $\delta(\alpha_m) = 0$,
and it implies that $\ell(S^{(m)}) - \ell(S^{(m-1)}) < d(m)$.

Let us treat the case (ii).
In the same way as the case (i), 
if $B_0$ contains an element $t^x$ with $x < q$,
then there exists a non-negative integer $m$ such that
$\ell(S^{(m)}) - \ell(S^{(m-1)}) < d(m)$.
In fact, for the minimum number $m$ such that
$B_m$ contains no element $t^x$ with $x < q$,
fix an element $t^x$ of $B_{m-1}$.
Then for $t = \pi(t^x)$, 
we have $\delta(t) = 0$ and $\delta(\beta_m) = 1$ since
if $\delta(t) = 1$ and $\delta(\beta_m) = 0$ is true,
then there exists an element $0^x$ of $T(S)$ satisfying
that $t < 0^x < \beta_m$ by Proposition~\ref{PropOfTwoTerms}.
It implies that $B_m$ contains an element $0^x$,
and this is a contradiction.

By the above, in the case (i) and (ii),
we have $\ell(S') - \ell(S^{(0)}) < |A_0| + |B_0'|$,
and it follows that $\ell(S') < \ell(S) - 1$.
\end{proof}

\begin{corollary} \label{CoroOfrr+1}
If $r + 1 < q$, then the specialization is not generic.
\end{corollary}

\begin{proof}
For a small modification of $0^r$ and $1^q$,
by Lemma~\ref{LemOf01},
there exists an element $t^x$ of $T(S)$ such that
$0^r < t^x < 1^q$ and $r < x < q$.
If $\delta(t^x) = 0$, then the element $t^x$ belongs to $A_0$,
and the assertion follows from Proposition~\ref{PropOfrq}.
Let us see the case $\delta(t^x) = 1$.
If the set $B_0$ contains $t^x$,
then by Proposition~\ref{PropOfrq}
we complete the proof.
If $B_0$ does not contain $t^x$,
then we have $a \geq a'$,
where $a'$ is the minimum number satisfying that
$\alpha_{a'} = \beta_0$.
Then $|B_0| < |B_0'| + |I|$ holds,
where sets $B'_0$ and $I$ are same as the proof of 
Proposition~\ref{PropOfrq}.
Hence we see $\ell(S') < \ell(S) - 1$.
\end{proof}

Let $\xi = \sum_{i = 1}^z (m_i, n_i)$ be a Newton polygon.
Let $S$ be the ABS of the ${\rm DM_1}$ $N_\xi$.
Recall that the ABS $S$ is described as $S = \bigoplus_{i=1}^z S_i$
for ABS's $S_i$ corresponding to the ${\rm DM_1}$ $N_{m_i, n_i}$.
We say a full modification $S'$ is generic if $\ell(S') = \ell(S) - 1$.
Propositions~\ref{PropOfAVanish}~and~\ref{PropOfBVanish}
imply that
all generic specializations are given by full modifications.
Now let us show Theorem~\ref{ThmOfOneToOne}
which implies that to determine boundaries of $H(\xi)$
it is enough to deal with Newton polygons consisting of two segments.

\begin{proof}[Proof of Theorem~\ref{ThmOfOneToOne}]
Let $S$ be the ABS of $\xi$.
Let us construct a bijection map
$$\bigcup_{i=1}^{z-1}\{\text{generic specializations of }R_i\}
\longrightarrow 
\{\text{generic specializations of }S\},$$
where $R_i$ denotes the ABS of the two slopes Newton polygon 
$(m_i, n_i) + (m_{i+1}, n_{i+1})$.
By Corollary~\ref{CoroOfrr+1}, it suffices to show the claim:
{\it Let $r$ be a natural number with $r < z$.
The full modification of $S$ by $0^r_i$ and $1^{r+1}_j$ is generic
if and only if the full modification of $R_r = S_r \oplus S_{r+1}$ 
by the same $0^r_i$ and $1^{r+1}_j$ is generic.}
For the small modification by $0^r$ and $1^{r+1}$ of $S$,
we use the same notation as
Construction~\ref{ConstAn} and Construction~\ref{ConstBn}.
We give the proof of this theorem in 4 steps.

{\bf Step\,1:}
If the set $A_0$ contains an element $0^x$ with $x \neq r$,
then $x < r$ holds.
In fact, if $r < x$ is true, then
$0^x < \beta_0$, with $\beta_0 = 1^{r+1}$, in $T(S)$.
Thus we have $r + 1 < x$ and $0^x_{m_x+1} < 0^{r+1}_{m_r+1}$.
This contradicts with Proposition~\ref{PropOfTwoTerms}.
Similarly, if $1^y$ with $y \neq r + 1$ belongs to the set $B_0$,
we have then $r + 1 < y$.

{\bf Step\,2:}
Let $m$ be the minimum number such that $\alpha_m = 1^r_{m_r}$.
We claim that the set $A_m$ contains no element $t^x$ with $x \neq r$.
Suppose that the set $A_{m-1}$ contains an element $t^x$.
Since $1^x_{m_x} < 1^r_{m_r}$, we see $\delta(\pi(t^x)) = 0$,
and $\pi(t^x)$ does not belong to $A_m$.

{\bf Step\,3:}
Fix an element $0^x$ of $A_0$,
and let $n$ be the maximum number satisfying that 
$A_n$ contains $\pi^n(0^x)$.
Put $t = \pi^n(0^x)$.
If $\delta(\alpha_{n+1}) = 0$ and $\delta(\pi(t)) = 1$,
we have then $0^r_{m_r+1} < 0^x_{m_x+1}$ with $x < r$.
This is a contradiction.
Thus we see $\delta(\alpha_{n+1}) = 1$ and $\delta(\pi(t)) = 0$.

{\bf Step\,4:}
Put $R = S_r \oplus S_{r+1}$.
Consider the small modification of $R$ 
by the same elements $0^r$ and $1^{r+1}$ as above.
Let $C_0, C_1, \dots$ be sets obtained by the operation of
Construction~\ref{ConstAn},
and let $D_0, D_1, \dots$ be sets obtained by 
the operation of Construction~\ref{ConstBn}.
We have ABS's $R^{(0)}, R^{(1)}, \dots$ by these constructions.
Let $E$ (resp. $F$) denote the subset
of $A_0$ (resp. $B_0$) consisting of elements $0^x$ (resp, $1^y$) with $x \neq r$
(resp. $y \neq r + 1$).
Then we have $\ell(S) - \ell(S^{(0)}) = \ell(R) - \ell(R^{(0)}) + |E| + |F|$.
Moreover, by Step\,2 and Step\,3, 
we have $\ell(S') - \ell(S^{(0)}) = \ell(R') - \ell(R^{(0)}) + |E| + |F|$.
This completes the proof.
\end{proof}

\section{The case of Newton polygons consisting of two slopes} \label{TwoSlopesNP}

In this section, we give a proof of Theorem~\ref{ThmOfSpeNP}.
In Section~\ref{ArbitNP},
we have seen that it suffices to deal with Newton polygons 
consisting of two segments to classify boundary components of central streams.
The notation is as Section~\ref{Prelim} and Section~\ref{ArbitNP}.

For a Newton polygon $\xi$, we will give
a one-to-one correspondence between
$B(\xi)$ and $B(\xi^{\rm D})$.
See Section~\ref{Intro} \eqref{EqOfBxi} for the definition of $B(\xi)$.
Moreover, to get a bijection between $B(\xi)$ and $B(\xi^{\rm C})$,
we use Lemma~\ref{LemOfXiAndXiC}.


\begin{proof}[Proof of Theorem~\ref{ThmOfSpeNP}]
The assertion is paraphrased as follows:
Let $S$ (resp. $R$) denote the ABS associated to $\xi$ (resp. $\xi^C$).
The map from
a generic specialization $S'$ of $S$ obtained by the small modification 
by $0^1_i$ and $1^2_j$ to
the generic specialization $R'$ of $R$ obtained by the small modification 
by $0^1_i$ and $1^2_j$ is bijective.
The set $T(R)$ can be regarded as a subset of $T(S)$.
By Lemma~\ref{LemOfXiAndXiC}, we have
\begin{eqnarray}
\{(0^1, 1^2) \in T(S)^2 \mid 0^1 < 1^2 \text{ in }T(S)\}
= \{(0^1, 1^2) \in T(R)^2 \mid 0^1 < 1^2 \text{ in }T(R)\}.
\end{eqnarray}
Fix elements $0^1$ and $1^2$, with $0^1 < 1^2$, of $T(R) \subset T(S)$.
Consider the small modifications by these $0^1$ and $1^2$ for $S$ and $R$. 
Let $A_n$ (resp. $A_n'$) be sets obtained by Construction~\ref{ConstAn}
for the small modification by $0^1$ and $1^2$ in $S$ (resp. $R$).
Clearly we have $A_0 = A_0'$.
For a natural number $n$, suppose that $A_n = A'_n$.
If elements $t$ of $A_n$ satisfy that $\Delta(S)(t) = 0$,
then by Lemma~\ref{LemOfXiAndXiC} we see that $A_{n+1} = A'_{n+1}$. 
Moreover, 
if elements $t$ of $A_n$ satisfy $\Delta(S)(t) = 1$,
then it follows from Lemma~\ref{LemOfXiAndXiC} that 
$A_{n+2} = A'_{n+1}$ and $|A_{n+2}| = |A_{n+1}|$.
Similarly, for sets $B_n$ and $B'_n$ obtained by Construction~\ref{ConstBn},
we have $B_0 = B'_0$.
For a natural number $n$,
we suppose that $B_n = B'_n$.
Similarly as above, we have
\begin{eqnarray}
\begin{cases}
B_{n+1} = B'_{n+1} & \text{if } \Delta(S)(t) = 0 \text{ for } t \in B_n, \\
B_{n+2} = B'_{n+1} & \text{otherwise}.
\end{cases}
\end{eqnarray}
Moreover, for the latter case, we have $|B_{n+2}| = |B_{n+1}|$.
Thus we see that $\ell(S') = \ell(S) - 1$ if and only if $\ell(R') = \ell(R) - 1$.
This completes the proof.
\end{proof}

\begin{example} \label{ExOfSpe2532}
For the Newton polygon $\xi$ of Example~\ref{ExOfSpe2735},
we have $\Phi(\xi) = (2, 5) + (3, 2)$.
Let $S$ be the ABS associated to $\Phi(\xi)$.
For this ABS
\\
\vspace{6mm}
$$S = \xymatrix@=1pt
{
1_{ 1 }^ 1 \ar@/_30pt/[rrrrrrr] &
1_{ 2 }^ 1 \ar@/_30pt/[rrrrrrr] &
0_{ 3 }^ 1 \ar@/_30pt/[ll] &
0_{ 4 }^ 1 \ar@/_30pt/[ll] &
0_{ 5 }^ 1 \ar@/_30pt/[ll] &
1_{ 1 }^ 2 \ar@/_30pt/[rrrr] &
1_{ 2 }^ 2 \ar@/_30pt/[rrrr] &
0_{ 6 }^ 1 \ar@/_30pt/[llll] &
0_{ 7 }^ 1 \ar@/_30pt/[llll] &
1_{ 3 }^ 2 \ar@/_30pt/[rr] &
0_{ 4 }^ 2 \ar@/_30pt/[lllll] &
0_{ 5 }^ 2 \ar@/_30pt/[lllll]
},$$
\vspace{4mm}
\\
the specialization $S'$ obtained by exchanging $0^1_4$ and $1^2_2$ is 
\\
\vspace{6mm}
$$S' = \xymatrix@=1pt
{
1_{ 1 }^ 1 \ar@/_30pt/[rrrrrrr] &
0_{ 3 }^ 1 \ar@/_30pt/[l] &
1_{ 2 }^ 1 \ar@/_30pt/[rrrrrr] &
1_{ 2 }^ 2 \ar@/_30pt/[rrrrrr] &
0_{ 5 }^ 1 \ar@/_30pt/[lll] &
1_{ 1 }^ 2 \ar@/_30pt/[rrrrr] &
0_{ 4 }^ 1 \ar@/_30pt/[llll] &
0_{ 6 }^ 1 \ar@/_30pt/[llll] &
0_{ 7 }^ 1 \ar@/_30pt/[llll] &
0_{ 4 }^ 2 \ar@/_30pt/[llll] &
1_{ 3 }^ 2 \ar@/_30pt/[r] &
0_{ 5 }^ 2 \ar@/_30pt/[lllll]
}.$$
\vspace{4mm}
\\
For the small modification by $0^1_4$ and $1^2_2$,
we have sets
$A_0 = \{0^1_5\}$, $A_1 = \emptyset$,
$B_0 = \{1^2_1\}$ and $B_1 = \emptyset$.
Thus we have $a = 1$ and $b = 1$.
One can check that these $S$ and $S'$ satisfy $\ell(S') = \ell(S) - 1$.
Moreover, the sets of pairs $(0^1_i, 1^2_j)$ constructing generic specializations
for ABS's corresponding to $\xi$ and $\Phi(\xi)$ are both
$$\{(0^1_4, 1^2_1),\ (0^1_4, 1^2_2),\ (0^1_5, 1^2_1),\ (0^1_5, 1^2_2),\ 
(0^1_6, 1^2_3),\ (0^1_7, 1^2_3)\}.$$ 
\end{example}

\end{document}